\documentclass{article}
\usepackage[utf8]{inputenc}

\usepackage{bbm}
\usepackage{amsmath}
\usepackage{amsfonts}
\usepackage{amssymb}
\usepackage{amsthm}
\usepackage{graphicx}
\usepackage[space]{grffile}
\usepackage{caption}
\usepackage{subcaption}
\usepackage{float}
\usepackage{natbib}
\usepackage{geometry}

\newtheorem{theorem}{Theorem} 
\newtheorem{corollary}{Corollary}[theorem]
\newtheorem{lemma}[theorem]{Lemma}

\title{Blocked Gibbs Sampling for Improved Convergence in Finite Mixture Models}
\author{David Michael Swanson\footnote{ORCID: 0000-0003-3174-1656 \\
MSC2020 Subject Classifications. Primary 60B10; Secondary 65C05, 62H30.
}\\ \small{University of Texas MD Anderson Cancer Center} \\ \small{Department of Biostatistics, Houston TX}}
\date{October 2024}

\begin{document}

\maketitle

\abstract{Gibbs sampling is a common procedure used to fit finite mixture models.  However, it is known to be slow to converge when exploring correlated regions of a parameter space and so blocking correlated parameters is sometimes implemented in practice.  This is straightforward to visualize in contexts like low-dimensional multivariate Gaussian distributions, but more difficult for mixture models because of the way latent variable assignment and cluster-specific parameters influence one another.  Here we analyze correlation in the space of latent variables and show that latent variables of outlier observations equidistant between component distributions can exhibit significant correlation that is not bounded away from one, suggesting they can converge very slowly to their stationary distribution.  We provide bounds on convergence rates to a modification of the stationary distribution and propose a blocked sampling procedure that significantly reduces autocorrelation in the latent variable Markov chain, which we demonstrate in simulation.
\vspace{0.1cm}

}


\section{Introduction}
\label{intro}

Mixture models are ubiquitous in the field of statistics and are useful tools to model phenomena composed of diverse underlying data generating mechanisms.  They have been developed and well-studied over many years, and some understanding of convergence rates and consistency of mixing distributions and component parameters under both known and unknown numbers of components has been achieved \citep{miller_mixture_2018,richardson_bayesian_1997,miller_inconsistency_2014, robert_mixtures_1996,stephens_dealing_2000,fruhwirth-schnatter_finite_2006}.  Mixture models are often fit by augmentation of the data with latent variables  which specify component membership of an observation.  
Because one goal of fitting a mixture model can be deconvolution of the observed data into their underlying components, one likewise desires to understand their estimation and improve it if possible.  

Gibbs sampling is a common way to sample from the latent variable space in mixture models because of the convenience of its closed-form conditional distributions \citep{geman_stochastic_1984, bensmail_inference_1997, celeux_computational_2000,robert_monte_1998}.  A drawback of Gibbs sampling however is that it can be slow to converge when exploring correlated regions of a parameter space \citep{roberts_updating_1997, roberts_convergence_1998}.  Several authors have studied convergence rates of Gibbs sampling for this reason and in some cases tried to improve them \citep{diaconis_geometric_1991, liu_covariance_1995, roberts_updating_1997, hobert_improving_2011, amit_comparing_1991, roberts_simple_1994, sahu_convergence_1999}.  As a result, blocked Gibbs sampling approaches, or sampling from a joint conditional distribution rather than each parameter one at a time,
have been proposed for a variety of contexts including variance components models, Gaussian Markov random field models, and multivariate Gaussian distributions \citep{rosenthal_rates_1995, tan_block_2009, jensen_blocking_1995,jensen_blocking_1999,knorr-held_block_2002}.  In some cases, convergence rates are known and straightforward to compute, such as with the multivariate Gaussian distribution, and it has been shown that there is often benefit to using larger blocks during sampling \citep{roberts_convergence_1998, amit_comparing_1991}.  Intuitively, blocking improves convergence because it moves correlation from the Gibbs sampler to the random number generator \citep{seewald_discussion_1992}.  

Less often has blocking been proposed in the context of mixture models, and when it has it is from a non-parametric Bayesian perspective under different parameterizations \citep{ishwaran_gibbs_2001,ishwaran_bayesian_2001,argiento_blocked_2016}.  The blocked sampling is not performed directly using joint conditional distributions in the latent variable space, 
nor is the motivation the result of significant correlation among subsets of latent variables.  Unanswered questions therefore remain regarding correlation structure in the latent variable space, how that structure may influence convergence of the Markov chain generated under Gibbs sampling, and whether slow convergence rates can be calculated and then improved by blocked sampling.  





To attempt to  answer these questions, with caveats, we consider in this work the Gaussian finite mixture model with a known number of components within a Bayesian framework and show that latent variables exhibit correlation structure as a function of observations' proximities to component means and one another.  Because the presence of component parameters makes the study of the correlation of latent variables more difficult and clusterings are often of primary inferential interest, we integrate out mean and variance component parameters in calculating the joint distributions of latent variables needed to study the problem \citep{murphy_machine_2012,van_dyk_partially_2008,liu_collapsed_1994}. These calculations reveal that, in some cases, such as when observations are close to one another and deep in the tails between two or more component distributions, ie, outlier observations, that correlation can be considerable and is not bounded away from 1.  We relate this phenomenon to latent variables' stationary distributions, which suggests that a subset of latent variables can converge arbitrarily slowly, implying that standard Gibbs sampling-based estimation of their clustering will be poor.  In such cases, we propose blocking these correlated latent variables, or sampling from their joint distribution, conditional on the allocation of the complementary, non-outlier set of latent variables in the model so that convergence to the stationary distribution is not impeded by significant autocorrelation in the Markov chain.  In so doing, we introduce by necessity a modified notion of stationarity in this mixture model setting. 


The paper is arranged into nine sections.  In section 2, we describe the collapsed Bayesian Gaussian mixture model and then calculate expressions for the joint distribution of an arbitrary collection of latent variables conditional of the complementary set.  Next, we consider the two component mixture model and calculate correlation of two arbitrary latent variables as a function of cluster allocations and observations' proximities to one another in order to identify situations in which blocked Gibbs sampling is advantageous. In section 4, we introduce an alternative notion of convergence specific to subsets of latent variables then, in Section 5, consider more than two components in the mixture model and provide convergence rate bounds for the set of blocked, outlier observations.  
In sections 6 and 7, we address computational burden of the blocking strategy and then consider a wider variety of cluster spatial orientations that complicate correlation structures in the latent variable space.
In section 8, we show in simulation significantly reduced autocorrelation in the Markov chain and improved clustering estimation.  We conclude in section 9 with a discussion.  

\section{Collapsed Gaussian mixture model} 
\label{sec:conditional}

Suppose we want to cluster data ${\pmb Y} = (Y_1, \dots , Y_N)$ for $Y_i \in \mathbbm{R}^D$, which follow the mixture model 
$$L(Y_i|{\pmb \pi}, {\pmb \theta})=\sum_{k=1}^K \pi_k L_{\theta_k} (Y_i)$$
where ${\pmb \pi }= (\pi_1 \dots \pi_K) $ are weights summing to 1 and $\theta_k$ parameterizes component density $L_{\theta_k}$. By introducing a latent variable $C_i$ for every $Y_i$ that determines membership in one of the $K$ different component distributions, we can sample from ${\pmb C}=(C_1 , \dots ,C_N)$ given the data with  

$$L({\pmb Y }| {\pmb C}, {\pmb \theta} ) p({\pmb C} | {\pmb \pi}) =f({\pmb C},{\pmb Y} | {\pmb \pi}, {\pmb \theta}) \propto f({\pmb C}| {\pmb \pi}, {\pmb \theta}, {\pmb Y})$$
Considering ${\pmb \pi}$ and ${\pmb \theta}$ as random, one can introduce another level of hierarchy with $p({\pmb \pi} | \beta)$ and $p(\theta_k | \theta_0)$, $k \in \{1,\dots , K\}$, giving 


\begin{align*}
f({\pmb C},{\pmb Y }| {\pmb \pi}, {\pmb \theta} ) p({\pmb \pi} | \beta) p(\theta_1, \dots , \theta_K | \theta_0)  & = f({\pmb C},{\pmb Y},{\pmb \theta},{\pmb  \pi} | \beta, \theta_0)  \\ = \prod_{k=1}^K  f({\pmb C}_{\{k\}}, & {\pmb Y}_{\{k\}},  \theta_k ,{\pmb \pi} | \beta , \theta_0) \\ 
=\prod_{k=1}^K L( {\pmb Y}_{\{k\}} |  & {\pmb C}_{\{k\}},  \theta_k ) p(\theta_k | \theta_0 )  p( {\pmb C}_{\{k\}} | {\pmb \pi} ) p({\pmb \pi} | \beta ) 
\end{align*}
and where in the last equalities we partition the terms according to cluster allocations of ${\pmb C}$ with $\{k \}=\{j : C_j = k \} $ for ${\pmb C}_{\{ k \}}$.  
One can then integrate out ${\pmb \theta}$ and ${\pmb \pi}$ in order to consider correlation among the $C_i$ unobscured by these parameters.
We can leverage conjugacy by supposing that each $\theta_k = (\mu_k , \Sigma_k) \sim \mbox{GIW}(\mu_0,\kappa_0, \nu_0, \Sigma_0)$, the multivariate Gaussian Inverse Wishart distribution
\citep{murphy_machine_2012}.  If additionally ${\pmb \pi} \sim \mbox{Dirichlet}(\beta_1, \dots , \beta_K)$ with $\beta_k = \beta  \;\, \mbox{for all} \; k\in \{1,\dots , K \}$, we have  

\begin{align*}
 f({\pmb C} \, |\, \beta, \theta_0 , {\pmb Y} ) &  \propto \hspace{0.0cm} \\ 
 \hspace{1.6cm}  \Bigl(\prod_k &  \int_{\mu_k,\Sigma_k} L({\pmb Y}_{\{k\} }  | {\pmb C}_{\{ k \}}, {\bf \mu_k} , {\bf \Sigma_k}) p_\mu({\bf \mu_k}, {\bf \Sigma_k}  | \mu_0 , \Sigma_0 ) d\mu_k d\Sigma_k  \Bigr) \nonumber \\
&  \hspace{4cm} \cdot \Bigl( \prod_k \int_{\pi}  p( {\pmb C}_{\{k\}} |  {\pmb \pi} ) p({\pmb \pi} | \beta ) d\pi \Bigr) \nonumber \\
\end{align*}
\begin{align}
 =  &  \biggl( \prod_k \frac{\kappa_0^{D/2} |\Sigma_0|}{(\kappa_0 + n_k)^{D/2}|S_{\{k\}}|^{\nu_0/2 + n_k/2} }  G_k \biggr)  \cdot \biggl( \frac{\Gamma (K\beta)}{\Gamma(N+K\beta) \Gamma (\beta)^K} \prod_k \Gamma (\beta + n_k )   \biggr) \\
  &  \hspace{4.4cm} \propto \prod_k \frac{\Gamma (\beta + n_k )  \cdot \Pi_{i=1}^D \Gamma (\nu_0 + n_k+1-i) }{|S_{\{k\}}|^{\nu_0/2 + n_k/2} \; \cdot \; (\kappa_0 + n_k)^{D/2} } 
  \label{eqn:marg}
\end{align}
where we define $n_k = | \{ k \} | $,  $\theta_0 = (\mu_0,\kappa_0, \nu_0, \Sigma_0)$,
and $S_{\{k\}}$ and $G_k$ are

\begin{align}
S_{\{k\}} = \Sigma_0 + S_{{Y}_{\{k\}}} + \frac{\kappa_0 n_k}{\kappa_0 + n_k}(\overline{Y}_{\{ k\}}-m_0)(\overline{Y}_{\{ k\}}-m_0)^T
\label{eqn:sum_squares}
\end{align}

$$G_k = \prod_{i=1}^D \frac{\Gamma (\nu_0 + n_k +1-i)}{\Gamma (\nu_0 +1-i)}$$
where $S_{{Y}_{\{k\}}}$ is the centered sum of squares of ${\pmb Y}_{\{ k\}}$ 
and ${\overline Y_{\{k\}}}$ is the mean of ${\pmb Y}_{\{k\}}$.  When the argument of $| \cdot | $ is a set it denotes cardinality while if a matrix it is the determinant.  



%
Define $\{b \} = \{i_1,i_2,\dots , i_B\}$ for some $B<N$, a blocking set of indices, where ${\pmb C}_{\{ b \}}$ is the subset of ${\pmb C}$ with indices in $\{b\}$, and ${\pmb C}_{\backslash \{ b \}} =  {\pmb C} \, \backslash \, {\pmb C}_{\{b \}}$.  This index set is generic now but will ultimately represent the indices of observations for which blocking poses a convergence rate advantage under Gibbs sampling, in particular the indices of outlier observations between two or more components.  Also define $\{k \backslash b\}  = \{j : C_j = k\} \backslash \{b\}$, a set whose size we call $n_{k\backslash b}$. Then we can calculate  
\begin{align}
\begin{split}
  & f({\pmb C}_{\{b\}} \, |\, {\pmb C}_{\backslash \{ b \}} ,  \beta,  \, \theta_0, {\pmb Y}  ) =  \frac{f({\pmb C} | \beta, \theta_0, {\pmb Y}  )}{f( {\pmb C}_{\backslash \{ b \}} |  \beta, \theta_0, {\pmb Y} )}    \\ 
   & \hspace{-0.1cm} \propto \prod_k  \frac{\;\; \Gamma (\beta + n_k )   \;\; |S_{\{k\backslash b\}}|^{\nu_0/2 + n_{k\backslash b}/2} \;\; (\kappa_0 + n_{k\backslash b})^{D/2} \;\; \Pi_{i=1}^D  \Gamma (\nu_0 + n_{k}+1-i) }{\Gamma (\beta + n_{k\backslash b} ) \; |S_{\{k\}}|^{\nu_0/2 + n_k/2} \; (\kappa_0 + n_k)^{D/2} \;\;\;  \Pi_{i=1}^D   \Gamma (\nu_0 + n_{k\backslash b}+1-i) }  
\label{eqn:joint}
\end{split}
\end{align}
where $S_{\{k \backslash b\}}$ 
is defined analogously to $S_{\{k\}}$ but with the index set $\{k\backslash b \}$ and is the centered sum of squares of the $Y_{\{k \backslash b \}}$.
Notice that while subscripting random variables ${\pmb C}$ or ${\pmb Y}$ with a set like $\{k \}$ denotes the elements corresponding to that index set, in contrast subscripting $S$ with a set denotes a matrix.

In Equation (\ref{eqn:joint}), we see that if hyperparameters $\beta$, $\kappa_0$, and $\nu_0$ are positive,
it is only the fraction of determinants term $(|S_{\{k\backslash b\}}|^{\nu_0/2 + n_{k\backslash b}/2})$/$(|S_{\{k\}}|^{\nu_0/2 + n_{k}/2})$ that may not be bounded away from 0.  It is this feature that drives what we will find is the unbounded correlation of latent variables for observations in multiple tails of component distributions, conditional on the allocation of the complementary set of latent variables.  It is this term that will dominate in the circumstances of interest to us and so will be our focus going forward.


\section{Conditional Correlation in the latent variable space with K=2, B=2}
\label{sec:correlation}

Consider for simplicity the case with $K=2$ and $B=2$ so that the mixture model consists of two components and with a block size of two.   Suppose we want to calculate the correlation of $C_i$ and $C_j$, that is $\{b \}= \{i,j \} $, latent variables of the two observations $Y_i$ and $Y_j$, conditional on the allocation of the complementary set, denoted ${\pmb C}_{\backslash i,j}$. 


Because $B=2$ and there are $K^B$ cluster allocations for the blocking set, there are four possible assignments for ${\pmb C}_{\{ b \}}=\{C_i, C_j \}$.  We can denote these possibilities with probabilities $p_{11}$, $p_{12}$, $p_{21}$, and $p_{22}$, where for $l,m\in \{1,2 \}$, $p_{lm}$ 
is the probability $Y_i$ and $Y_j$ are allocated (equivalently, $C_i$ and $C_j$ are assigned) to components $l$ and $m$, respectively.  Since $S_{\{ k\backslash b \}}$ and $n_{k\backslash b}$ are constant in Equation (\ref{eqn:joint}) across $p_{11}$, $p_{12}$, $p_{21}$, and $p_{22}$, we can calculate the joint distribution with this expression and then normalize so the values sum to one:

\begin{align*}
f({\pmb C}_{\{i,j\}} \, | \, {\pmb C}_{\backslash \{ i,j \}} ,  \beta,  \, \theta_0, Y  ) \propto \prod_k  \frac{\;\; \Gamma (\beta + n_k )   \;\;  \Pi_{i=1}^D  \Gamma (\nu_0 + n_{k}+1-i) }{|S_{\{k\}}|^{\nu_0/2 + n_k/2} \; (\kappa_0 + n_k)^{D/2}  }
\end{align*}
where the expression is a function of ${\pmb C}_{\{ij\}}$ via $S_{\{k\}}$ and $n_k$, in the former case by way of the sum of squares calculation exhibited in Equation (\ref{eqn:sum_squares}).

 
 We introduce another piece of notation, $\{b_k\} = \{n : n\in \{ b\} \mbox{ and } C_n=k \} = \{b \} \cap \{ k \}$, indices of the blocking set allocated to component $k$, and which notice is specific to any point in the joint distribution of $C_i$ and $C_j$, that represented by $p_{lm}$ with $l,m\in \{1,2 \}$.
 For this reason, when generalizing to the $K\geq 2$ component mixture model in Section \ref{sec:converge_rate} below, $\{b_k\}$ is subscripted by $``lm"$, signifying the allocation of $Y_i$ and $Y_j$ to the $l$ and $m$ components, respectively (ie, $C_i=l$ and $C_j=m$, meaning $i\in \{b_l\}$ and $j\in \{b_m\}$).  

 
 

With these definitions, we assert the following decomposition of the sum of squares $S_{{Y}_{\{k\}}}$:
 

\begin{align*}
S_{{Y}_{\{k\}}}  = S_{{Y}_{\{k \backslash b \}}} +(n_k -  n_{k\backslash b})\frac{n_{k\backslash b}}{n_k} & (\overline{Y}_{\{ k\backslash b\}} \\
- \overline{Y}_{\{b_k\}})& ({\overline Y}_{\{ k\backslash b\}}  - \overline{Y}_{\{b_k\}})^T +V_{b_k}
\end{align*}
where 

\begin{align*}
V_{b_k}=({\pmb Y}_{\{b_k\}}-\overline{Y}_{\{b_k\}})({\pmb Y}_{\{b_k\}}-\overline{Y}_{\{b_k\}})'
\end{align*}
and where ${\pmb Y}_{\{b_k\}}$ is understood as the $D \times |\{b_k\}|$ matrix whose observations are oriented vertically, and from which the $D \times 1$ mean vector $\overline{Y}_{\{b_k\}}$ will be subtracted one by one column-wise.  So the sum of squares for component $k$ can be expressed as the sum of that without the blocking set and the scaled, squared difference between the means of observations allocated to component $k$ in and not in the blocking set, plus the $V_{b_k}$ term (the ``within'' sum of squares of the blocked observations allocated to component $k$). Note that for those $k$ for which $\{b_k\}$ is the empty set, $n_k - n_{k\backslash b}$ is 0 so that there is no perturbation to $S_{{Y}_{\{k \backslash b\}}}$.  
We decompose $S_{{Y}_{\{k\}}}$ this way for convenience to study the correlation of $C_i$ and $C_j$ as a function of increasing $d_k$ for $(d_k \, \overline{Y}_{\{ k\backslash b\}} -  \overline{Y}_{\{b_k\}})$ for $k\in \{1,2 \}$ when $K=2$, the mean distance between the blocked observations allocated to component $k$ and observations in component $k$ without the blocked observations. 


For analytic tractability in what follows, assume for the GIW prior that $\Sigma_0$ is positive definite and a flat prior on the mean, ie $\kappa_0=0$.  The latter assumption gives 
$$S_{\{ k\}} -S_{\{ k\backslash b \}}=S_{{Y}_{\{k\}}} -  S_{{Y}_{\{k \backslash b \}}}$$
where the left hand side is Equation (\ref{eqn:sum_squares}) perturbed with and without the blocking set of $Y_{\{b \}}$, and the right hand side is the perturbed centered sum of squares.  
However, since the size of the blocking set allocated to component $k$, $n_k - n_{k\backslash b}$, will tend to be small relative to $n_{k\backslash b }$, $S_{\{ k\}} -S_{\{ k\backslash b \}} \approx S_{{Y}_{\{k\}}} -  S_{{Y}_{\{k \backslash b \}}}$ in the general case.  
Examination also shows that when $d_k\rightarrow \infty$, the case of interest for us, the dominant term $\overline{Y}_{\{k \backslash b\}}^2$ from both of the perturbed terms in Equation (\ref{eqn:sum_squares}) yields $(S_{{Y}_{\{k\}}} - S_{{Y}_{\{k \backslash b \}}} )/(  S_{\{ k\}} -S_{\{ k\backslash b \}}) \rightarrow 1$ for $n_{k\backslash b} \rightarrow \infty$ while $n_k- n_{k\backslash b}  = {\mathcal O}(1)$.

Assume for convenience and without loss of generality that the data ${\pmb Y}$ are translated so that the mean of $Y_i$ and $Y_j$
is $\overline{Y}_{\{ ij \}}=0$.  Examining correlation by scaling ${\overline {\pmb Y}}_{\{ k\backslash b \}}$ with $d_k$, $k\in \{1,2 \}$, maintains the orientation of the two sets of component means while increasing the distance between them. A fixed distance is then maintained between $Y_i$ and $Y_j$ while effectively pushing them increasingly into the component tails because the component means, $\overline{\pmb {Y}}_{\{ 1\backslash b \}}$ and ${\overline{\pmb Y}}_{\{ 2\backslash b \}}$, are moving away from $Y_i$ and $Y_j$. 
  
Call $p_{1\cdot}$ the marginal probability that $C_i$ is allocated to component 1, call $p_{\cdot 1}$ the marginal probability that $C_j$ is allocated to component 1, and define $p_{2 \cdot}$ and $p_{\cdot 2}$ analogously for component 2, so $p_{l\cdot} = p_{l1} + p_{l2}$ and $p_{\cdot m } = p_{1m} + p_{2m}$.  We can then calculate the joint distribution of $C_i$ and $C_j$ in the two component mixture model with





\begin{align*}
p_{11} = \Gamma_{11} \, \Bigl| S_{{\{1 \backslash b \}}} + 2\, \Bigl(\frac{n_{1\backslash b}}{n_{1\backslash b}+2}\Bigl) (d_1 \, \overline{Y}_{\{ 1\backslash b\}})(d_1 \, \overline{Y}_{\{ 1 \backslash b\}} &  )^T  + V_{\{ ij \}} \Bigr|^{-(n_{1\backslash b}+2)/2} \\ 
& \cdot \Bigl| S_{{\{2 \backslash b \}}} \Bigr|^{-({n_{2\backslash b}})/2}
\end{align*}

\begin{align*}
p_{12}  = \Gamma_{12} \, \Bigl| S_{{\{1 \backslash b \}}} + & \Bigl(\frac{n_{1\backslash b}}{n_{1\backslash b}+1}  \Bigl) ( d_1\, \overline{Y}_{\{ 1\backslash b\}} - {Y}_{i})( d_1\, \overline{Y}_{\{ 1\backslash b\}} -  {Y}_{i})^T  \Bigr|^{-(n_{1\backslash b}+1)/2} \\ 
 \cdot  \Bigl| S_{{\{2 \backslash b \}}}& +  \Bigl(\frac{n_{2\backslash b}}{n_{2\backslash b}+1}\Bigl) (d_2 \, \overline{Y}_{\{ 2\backslash b\}} -  {Y}_{j})(d_2\, \overline{Y}_{\{ 2\backslash b\}} -    {Y}_{j})^T  \Bigr|^{-(n_{2\backslash b}+1)/2}
\end{align*}


\begin{align*}
p_{21} =  \Gamma_{21} \, \Bigl| S_{{\{1 \backslash b \}}} + & \Bigl(\frac{n_{1\backslash b}}{n_{1\backslash b}+1}  \Bigl) ( d_1\, \overline{Y}_{\{ 1\backslash b\}} - {Y}_{j})( d_1\, \overline{Y}_{\{ 1\backslash b\}} -  {Y}_{j})^T  \Bigr|^{-(n_{1\backslash b}+1)/2} \\ 
 \cdot  \Bigl| S_{{\{2 \backslash b \}}}& +  \Bigl(\frac{n_{2\backslash b}}{n_{2\backslash b}+1}\Bigl) (d_2 \, \overline{Y}_{\{ 2\backslash b\}} -  {Y}_{i})(d_2\, \overline{Y}_{\{ 2\backslash b\}} -    {Y}_{i})^T  \Bigr|^{-(n_{2\backslash b}+1)/2}
\end{align*}

\begin{align*}
p_{22} =  \Gamma_{22} \, \Bigl| S_{{\{1 \backslash b \}}}  \Bigr|^{-({n_{1\backslash b}})/2} & \\
 \cdot  \Bigl| S_{{\{2 \backslash b \}}} + 2\,   \Bigl(\frac{n_{2\backslash b}}{n_{2\backslash b}+2} &  \Bigl)    (d_2 \, \overline{Y}_{\{ 2\backslash b\}})(d_2 \, \overline{Y}_{\{ 2 \backslash b\}}   )^T  + V_{\{ ij \}} \Bigr|^{-(n_{2\backslash b}+2)/2} 
\end{align*}
where we define the constant $\Gamma_{lm}$ with
%
%
\begin{align}
\Gamma_{lm} = \gamma \prod_{k=1}^2 \Bigl[ \Gamma (\beta + & n_{k\backslash b} +I_{k=l} +   I_{k=m} ) \cdot  \nonumber \\ 
& \frac{\Pi_{i=1}^D    \Gamma (\nu_0 + n_{k\backslash b} +I_{k=l} + I_{k=m}+1-i) }{  (\kappa_0 +  n_{k\backslash b} +I_{k=l} + I_{k=m}  )^{D/2} }  \biggr]
\label{eqn:constant}
\end{align}
and where recall that $p_{lm}$ denotes $C_i$'s allocation to component $l$ and $C_j$'s allocation
to component $m$.  
%
%
%
%
Above, we use $\overline{Y}_{\{ ij \}}=0$ by assumption and $V_{\{ ij \}}$ is the centered sum of squares of $Y_i$ and $Y_j$, and $\gamma$ in Equation (\ref{eqn:constant}) is a normalizing constant so that $p_{11}+p_{12}+p_{21}+p_{22}=1$, whose form we omit for brevity.  Then we have $Cor(C_i,C_j) = \bigl(p_{11}-(p_{1\cdot} )(p_{\cdot 1} )\bigr) / \sqrt{(p_{1\cdot}\cdot p_{2\cdot})(p_{\cdot 1 }\cdot p_{\cdot 2})}$.  There is no $V_{\{ ij \}}$ term in $p_{12}$ and $p_{21}$ because each component consists of only 1 additional observation, $Y_i$ or $Y_j$, so the within component sum of squares is 0. 


\begin{theorem}
Consider the joint distribution of $C_i$ and $C_j$ given the complementary allocation ${\pmb C}_{\backslash i,j}$ defined above with probabilities $p_{11},\, p_{12},\, p_{21},\, p_{22}$.  Provided $S_{{\{k \backslash b \}}}$ are non-singular, $n_{k\backslash b}>0$, $k=1,2$, then as $d_k \rightarrow \infty$ with $d_1 = {\mathcal O}(d_2)$, $Cor(C_i,C_j)\rightarrow 1$. 
\label{theorem_cor}
\end{theorem}

All proofs omitted in the main text can be found in the Supplementary Material.  Briefly for Theorem \ref{theorem_cor}, one shows that $p_{21}$ and $p_{12}$ are both $o(p_{11})$ and $o(p_{22})$ as the $d_k$'s become large.  And while it may be $p_{22}= o(p_{11})$ or $p_{11}= o(p_{22})$ depending on if $n_{1\backslash b} < n_{2\backslash b}$ or $n_{1\backslash b}> n_{2\backslash b}$, respectively, one still has $Cor(C_i,C_j)\rightarrow 1$.  That $p_{11} = {\mathcal O}(p_{22})$ or $p_{11} = o(p_{22})$ holds
reflects how increasingly well-leveraged outliers, $Y_i$ and $Y_j$, ``pull'' component means, ${\overline {\pmb Y}}_{\{ 1\backslash b \}}$ and ${\overline {\pmb Y}}_{\{ 2\backslash b \}}$, towards themselves synergistically when being allocated into that component, the influence of which becomes a function of the relative sizes of $n_{1\backslash b}$ and $n_{2\backslash b}$.  For clusters with smaller sample sizes, that ``pulling'' is increasingly effective as the distance between outliers and component means increases.  Thus, upon allocation into the component, the component's tail density at which the outlier is evaluated is significantly greater than it is with clusters of larger sample sizes.  This is all with respect to the diagonal components in the cross tabulation of $p_{11}$, $p_{12}$, $p_{21}$, and $p_{22}$, one can envision for the joint distribution of $C_i$ and $C_j$--those table elements representative of $Y_i$ and $Y_j$ being allocated to a cluster together. However, because the off-diagonal elements, representative of $Y_i$ and $Y_j$ being allocated to different clusters, become small still faster relative to the diagonal in the table above, correlation of $C_i$ and $C_j$ still goes to 1 even if one component along the diagonal dominates in the limit.    




\begin{corollary}
For $d_{1}$ fixed while $d_{2}\rightarrow \infty$, then $Cor(C_i,C_j)\rightarrow 0$.  
\label{corollary_cor}
\end{corollary}

\noindent
In fact, there exist slightly weaker conditions on $d_1$, which can be taken to $\infty$ at a slower rate than $d_2$, that rate a function of $n_{1\backslash b}$.  However, we focus on the fixed case here for its relative clarity and relevance.


One can also investigate this behavior numerically.  In Figures \ref{fig:equidistant} and \ref{fig:diff}, we see the correlation of two latent variables when their respective observations lie between two component means as a function of the distances from those means.  Figure \ref{fig:equidistant} shows a peak of correlation when the observations are equidistant between the two components and of equal value (ie, $Y_i=Y_j$), with a fast decline to 0 when moving closer to one or the other component mean.  Figures \ref{fig:diff}a and \ref{fig:diff}b are two perspectives on the same surface plot, where the distance between the two component means is fixed, but the observations may not be of the same value, and correlation is shown as a function of each one's distance to one of the component means.

To generate these figures, 48 total observations were used, and the centered sums of squares for either component for all but the two observations being blocked was 80.  For Figure \ref{fig:equidistant}, the range of either axis is 0.5 to 10 (ie, the point at which correlation between latent variables is highest is when both observations are 10 units from each component mean).  Whereas for Figure \ref{fig:diff}, the fixed distance between the two components means is 24 units and axes show the range 7 to 17 units (eg, along the diagonal where correlation is highest, the observations are of equal value, and vary between these two means, and at their midpoint 12 units from both).




\begin{figure}[H]
\begin{center}
\includegraphics[width=0.65\textwidth]{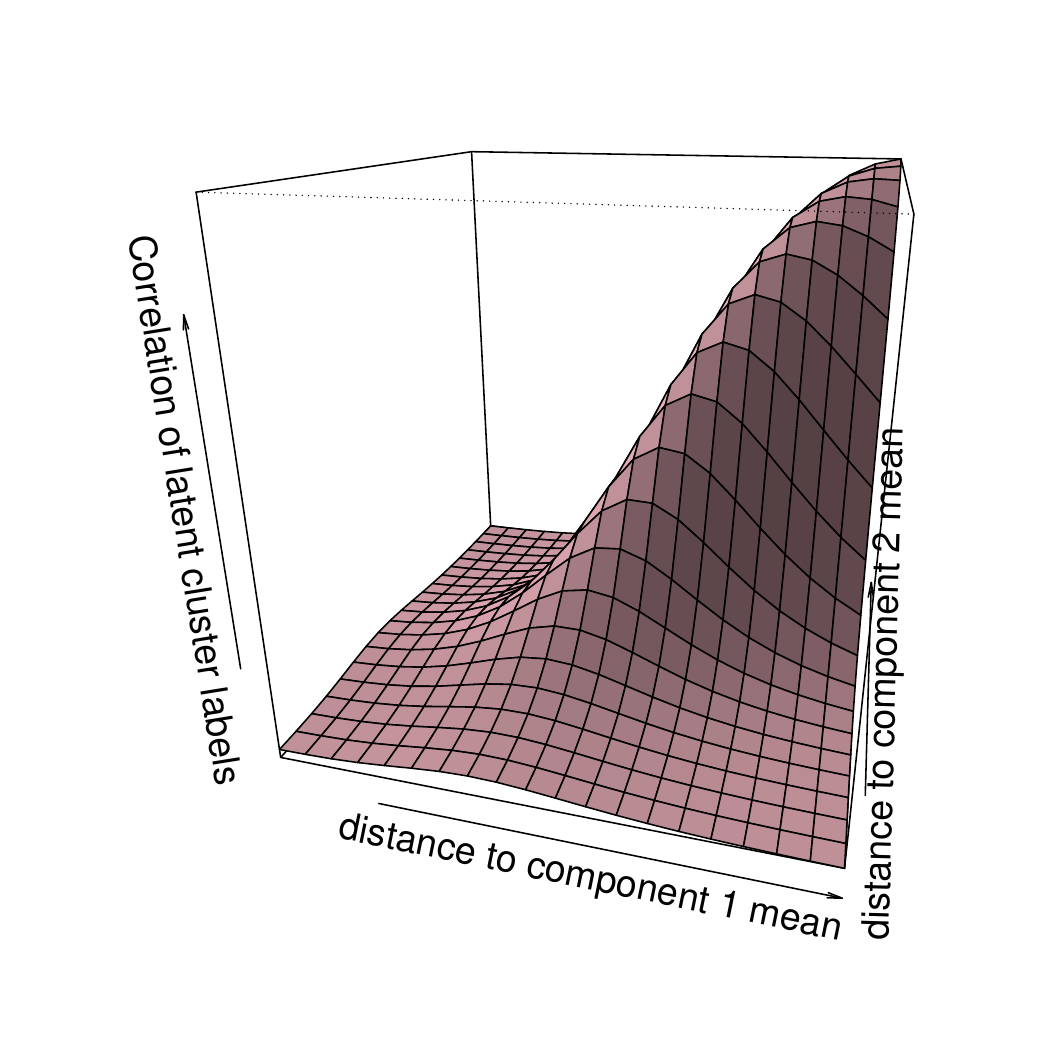} 
\captionsetup{width=0.9\textwidth}
\caption{ {\small A surface plot of correlation between $C_i$ and $C_j$, latent variables of $Y_i$ and $Y_j$ which are of equal value, as a function of their placement  between two component means.  The two axes are the distances from $Y_i=Y_j$ to the two component means.  The increasing ridge along the diagonal represents $Y_i$ and $Y_j$ equidistant between the two components and moving increasingly into their tails. }}
\label{fig:equidistant}  
\end{center}
\end{figure}


\begin{figure}[H]
\centering
\begin{subfigure}{0.45\textwidth}
\includegraphics[width=\linewidth]{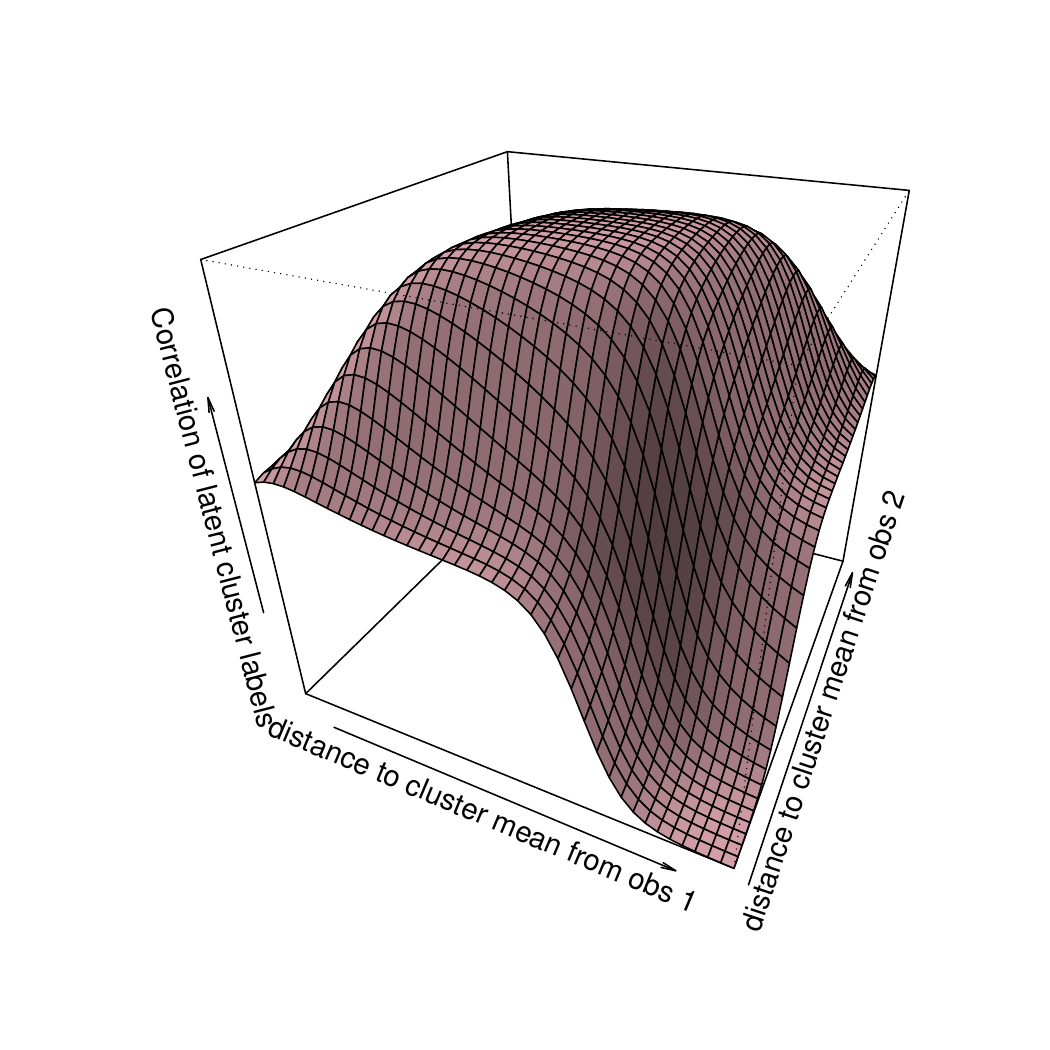}
\caption{} 
\label{fig:diff1}
\end{subfigure}
\hspace*{0.25cm} 
\begin{subfigure}{0.45\textwidth}
\includegraphics[width=\linewidth]{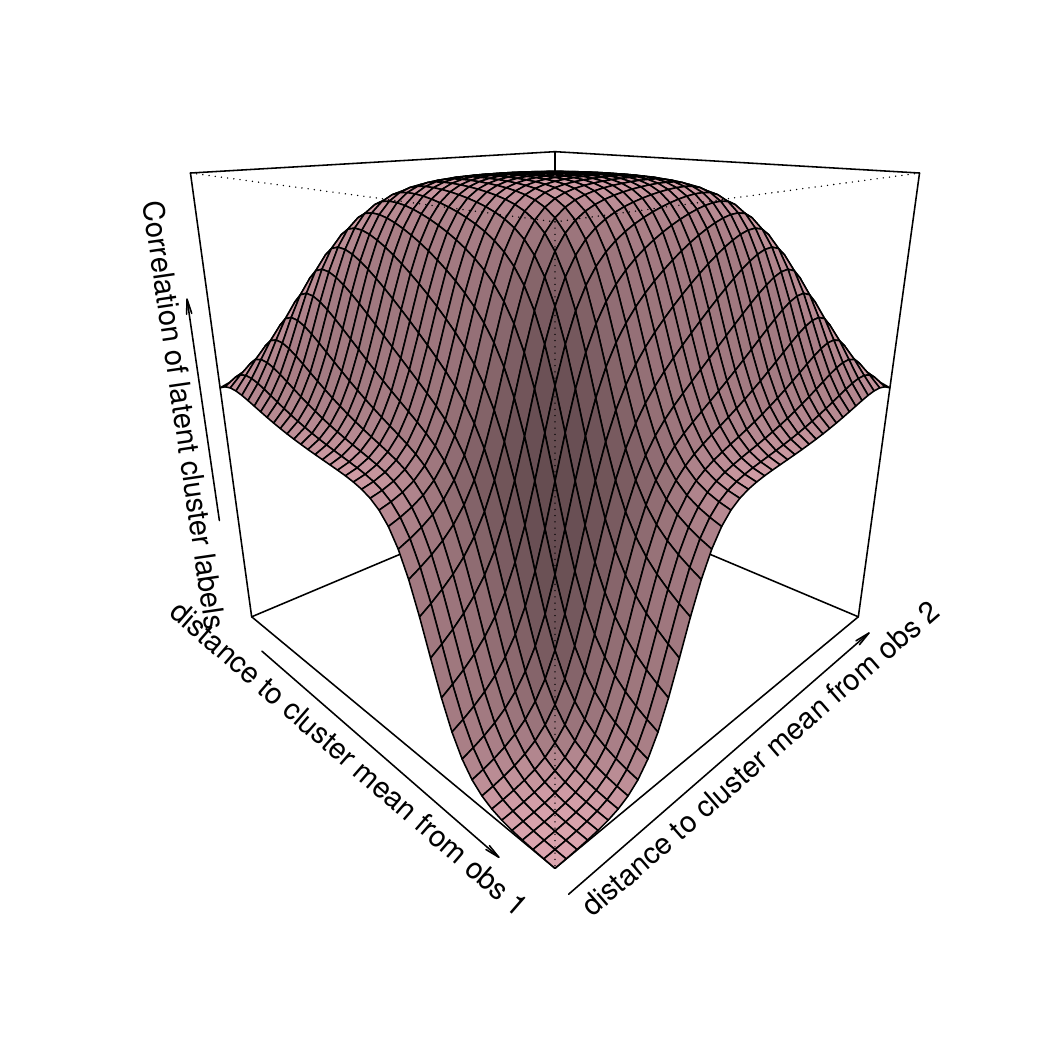}
\caption{} 
\label{fig:diff2}
\end{subfigure}
\captionsetup{width=0.9\textwidth}
\caption{{\small Two perspectives on the same correlation surface of $C_i$ and $C_j$, latent variables of $Y_i$ and $Y_j$, as a function of their placement  between two components.  The two axes are the distances from $Y_i$ and $Y_j$ to one of those component means, with the distance between the two components held constant.  The diagonal across the top of the surface (that closest to the silhouette of the surface) corresponds to $Y_i$ and $Y_j$ of equal value, while the other diagonal (that going from front to back of the surface) corresponds to $Y_i$ and $Y_j$ mirroring one another across the cluster means' midpoint. }}
\label{fig:diff}
\end{figure}

The intuition for the correlation phenomenon is that when observations are allocated to a cluster together, their combined weight pulls the component mean towards themselves synergistically as compared to their allocation being divided between multiple components.  Because Gaussian tails move toward zero increasingly fast as measured by their hazard function and because observations can move arbitrarily far from means, these points' leverage increases quickly and without bound, driving the correlation to 1.  While here we examine the phenomenon when observations are close to one another and pull all component means to a common center,  in Section \ref{sec:synergism} below we consider the dynamic when only a subset of component means in combination with a set of latent variables exhibit synergistic behavior.

\section{Partitioning the posterior}
\label{sec:partition}

Because calculating the convergence rates of Gibbs sampling to a mixture model's stationary distribution requires the composition of $N$ transition matrices of dimension $K^N \times K^N$, which is untenable for even relatively small $N$ and $K$, and further complicated by cluster label switching (\cite{stephens_dealing_2000,fruhwirth-schnatter_markov_2001}), we posit a background of clustering in the data such that some ``clear'' allocation of many latent variables constitutes $1-\epsilon$ for small $\epsilon$ of the posterior mass of the sampled model.

This is to assume we can partition the posterior space into a high probability region for which there is little question about the allocation of this subset of observations,
and one then considers independently the convergence rate of some complementary set of observations, call them ``outliers,'' or observations far and approximately equidistant from two or more component means--if they were not approximately equidistant, they would have a clear allocation to the nearest component and likewise not exhibit significant correlation with other latent variables of similarly-valued observations as demonstrated in Figures \ref{fig:equidistant} and  \ref{fig:diff}.  These outliers exist in the tails between component distributions where we have identified a high degree of correlation in the space of latent variables, and it is only these latent variables whose allocation is unclear and whose Gibbs sampling we seek to improve.  
Use as the partition  $ {\pmb  C} =   {\pmb C}_{ \{ b \}} \cup {\pmb  C}_{\backslash \{ b \}} $ such that ${\pmb  C}_{\backslash \{ b \}}$ are those latent variables whose allocation is relatively clear and equal to $c_{\backslash \{ b \}}$, and ${\pmb C}_{\{b \}}$ are the latent variables associated with outliers ${\pmb Y}_{\{b\}}$, that is,
\begin{align}
\sum_{{\pmb C}_{\{ b \}} \in {\mathcal C_{ \{ b \}  }}} f( {\pmb C}_{\{ b \}  }  , {\pmb C}_{\backslash \{ b \}}=c_{\backslash \{ b \}} | \beta, \theta_0, {\pmb Y} ) \nonumber \\
= f(  {\pmb C}_{\backslash \{ b \}}= & c_{\backslash \{ b \}} |  \beta, \theta_0, {\pmb Y}  ) = 1- \epsilon
\label{eqn:partition_space}
\end{align}
where ${\mathcal C}_{\backslash \{ b \}}$ is the set of all allocations of ${\pmb C}_{\backslash \{ b \}}$, a set of size $K^{N- B}$.  
Ideally, to well understand the clustering of data we would estimate 
$$(\pi^{C_i}_1, \pi^{C_i}_2, \cdots , \pi^{C_i}_K)$$
for each $C_i$ where
$$\sum_{{\pmb C}_{\backslash i} \in {\mathcal C_{\backslash i}}} f(C_i=k, {\pmb C}_{\backslash i} |  \beta, \theta_0, {\pmb Y}  ) =P(C_i=k |  \beta, \theta_0, {\pmb Y}  ) = \pi^{C_i}_k$$
for analogously defined ${\mathcal C_{\backslash i}}$, the set of all allocations of $C_{\backslash i}$, a set of size $K^{N-1}$.
However, here we study and try to improve convergence to 
$$(\tilde{\pi}^{C_i}_1, \tilde{\pi}^{C_i}_2, \cdots , \tilde{\pi}^{C_i}_K)$$
for those $i \in \{ b\}$, where 
\begin{align*}
\sum_{{\pmb C}_{\{b \}  \backslash i} \in {\mathcal C_{ \{b \}  \backslash i}}} f(C_i=k, {\pmb C}_{\{ b \} \backslash i }  & |  {\pmb C}_{\backslash \{ b \}}=c_{\backslash \{ b \}},  \beta, \theta_0, {\pmb Y}  ) \\ 
=f(C_i=k | & {\pmb C}_{\backslash \{ b \}}=c_{\backslash \{ b \}} ,  \beta, \theta_0, {\pmb Y}) = \tilde{\pi}^{C_i}_k
\end{align*}
a quantity in which we condition on ${\pmb C}_{\backslash \{ b \}}=c_{\backslash \{ b \}}$, useful because $\tilde{\pi}_k^{C_i} \in \bigl((1-\epsilon)\,\pi_k^{C_i}, \pi_k^{C_i} + \epsilon \, (1-\pi_k^{C_i})\bigr)$ for each $i$.  

One could alternatively motivate this approach by envisioning the component Gaussian distributions with clipped tails, which only overlap in some outlier region and for which the set of observations ${\pmb Y}_{\backslash \{ b\}}$, those not categorized as outliers, fall only in one such component distribution.  
We can leverage existing understanding of Gibbs sampling \citep{liu_covariance_1994,liu_covariance_1995, sahu_convergence_1999, roberts_updating_1997} to examine convergence of ${\pmb C}_{\{ b\}}$ to ${\pmb {\tilde \pi}}^{C_{\{b \}}}$, 
the joint distribution of ${\pmb C}_{\{b \}}$ 
conditional on ${\pmb C}_{\backslash \{ b \}}=c_{\backslash \{ b \}}$.

\section{Bounding convergence rates}
\label{sec:converge_rate}

\subsection{$K=2$, $B=2$ case}

Define $\{ {\pmb C}_{\{b\}}^{(0)}, {\pmb C}_{\{b\}}^{(1)}, \dots  \}$ a Markov chain generated by Gibbs sampling on ${\pmb C}_{\{ b\}}$, which follow the stationary distribution ${\pmb {\tilde \pi}}^{C_{\{b \}}}$ given the ``clear'' allocation ${\pmb C}_{\backslash \{ b \}}=c_{\backslash \{ b \}}$ of non-outlier observations, ${\pmb Y}_{\backslash \{ b \}}$.  
To make statements on convergence rates analytically simpler, we construct the chain as reversible so that after first sampling $C_{i_1} | {\pmb C}_{\{b \backslash i_1 \}}$, $C_{i_2} | {\pmb C}_{\{b \backslash i_2 \}}$, etc, in ascending order up to $i_B$ for a complete, unique set of indices (a forward pass), we proceed to sample the same indices in descending order per the technique of \citet{fill_eigenvalue_1991} (cf, \citet{liu_covariance_1995,roberts_updating_1997,amit_comparing_1991}), completing the sequence (a backward pass). We use these same orderings as Gibbs sampling proceeds.



Here we continue examination of the case where $K=2$, $B=2$, and the blocking indices are $\{ b \} = \{i, j \}$, and we consider the reversible Markov chain arising under Gibbs sampling.  One leverages $B=2$ so that $C_i$ and $C_j$ can be envisioned as Bernoulli, and so state the following lemma:
\begin{lemma}
For Gibbs sampling of binary $C_i$ and $C_j$ whose chain is reversible, 
$$\sup_{g,h} \; Cor(g(C_i), h(C_j))^2 = Cor(C_i, C_j )^2 = \rho_{\tilde{\pi}} ({\pmb C}_{\{i ,j\}})$$
\label{lemma_invariant}
\end{lemma}
\noindent
where $\rho_{\tilde{\pi}} (\cdot )$ is the convergence rate of its argument to ${\pmb {\tilde \pi}}^{C_{\{b \}}}$ or appropriate marginalization thereof according to the argument, and $g,h$ are functions on ${\mathbb R}$ under which $C_i,C_j$ are square integrable.

While the result bears resemblance to \citet{liu_covariance_1994}'s Theorem 3.2, the binary $C_i$ and $C_j$ notwithstanding, here the rate applies to a reversibilized chain, unlike the context in \citet{liu_covariance_1994} which consists of only forward passes of the Gibbs sampler.  The chain considered here therefore has twice as many samples for the same rate and so
reflects a convergence speed that is twice as slow.  This characteristic is consistent with the reversibilized chain updating twice each sample in a row, effectively and inefficiently not exploiting the new information presented by the first update of a component, sampling it a second time before finally moving on to updating the conjugate component.  We use the reversibilized chain to analyze the rate with greater clarity and will focus in what follows on improving that rate by other means.

If $Cor(C_i, C_j)^2$ is the convergence rate to ${\pmb {\tilde \pi}}^{C_{\{b \}}}$ for $\{b \}=\{i,j\}$,
Theorem \ref{theorem_cor} suggests it can be exceedingly slow with no bound away from 1 as observations move increasingly into the tails between two component means.  This is also suggested by Figures \ref{fig:equidistant} and \ref{fig:diff}
where 
correlation approaches 1 quickly and only modest movement into tails is necessary to yield a slow convergence rate for these observations.  This trait is problematic because it suggests that the clustering of observations close to one another and in these outlier regions will converge so slowly that Gibbs sampling may often mischaracterize cluster assignment of these data.  In practice, these observations will be allocated to one or another component, ``sticking" in that allocation and unable to move to an equally fitting one, related to high autocorrelation in the Markov chain.  Their joint membership in multiple components will be lost unless Gibbs sampling is performed for an indeterminate and possibly very long time.  Additionally, usual diagnostics may not detect the lack of convergence as these observations' contributions to the posterior space is relatively small.  And in many applications, including genomics and medicine, correct characterization of these highly leveraged outliers relative to typical, well-populated clusters is often very important.  

\subsection{$K>2$, $B=2$ case}

The $K=2$, $B=2$ case of two components and two observations is useful for exposition, but here we see can serve as a foundation for calculating a convergence rate lower bound in the general case.
While the exact convergence rate of ${\pmb C}_{\{ b\}}$ to ${\pmb {\tilde \pi}}^{C_{\{b \}}}$ would be useful, provision of a lower bound gives an indication of when one should consider performing more sampling or take a blocked sampling approach as we describe in Section \ref{sec:blocking}; ie, if the lower bound is high, one should consider blocked sampling, whereas if the lower bound is lower, it may not be necessary.

Suppose we have $K\geq 2$, $B=2$, then we can state:
\begin{lemma}
Consider having $K\geq 2$, $B=2$, and define $C_{i}^{k'}=I(C_i \leq k')$ and analogously for $C_{j}^{k'}$, indicators for allocation of $Y_i$ and $Y_j$ to the $1^{st}$ through $k'^{\,th}$ components with $1\leq k'<K$.   Then we have 
\begin{align*}
Cor(C_{i}^{k'},C_{j}^{k'})^2 = Cor&(I(C_i \leq k'), I(C_j\leq k'))^2 \\
\leq \sup_{g,h} \; & Cor(g(C_i), h(C_j))^2 = \rho_{\tilde{\pi}} ({\pmb C}_{\{i,j\}})
\end{align*}
for all $k'$, where $\rho_{\tilde{\pi}} ({\pmb C}_{\{i,j\}})$ is the convergence rate of Gibbs sampling with a reversible chain.
\label{lemma_multi_K}
\end{lemma}

\begin{proof}
The last equality follows 
from Lemma \ref{lemma_invariant}, and the inequality follows from the use of sup.  
\end{proof}

Practically, we can calculate $Cor(C_{i}^{k'},C_{j}^{k'})$ and therefore lower bound $\rho_{\tilde{\pi}}({\pmb C}_{\{i,j \}})$ by forming a $K \times K$ contingency table, call it ${\text {\bf U}}^{\{ ij \}}$, defining the joint distribution of $C_i$ and $C_j$, and partition it into a $2 \times 2$ contingency table with application of the indicator function $I(\cdot \leq k')$ to $C_i$ and $C_j$, $1\leq k'<K$.  In keeping with developed notation, element $l,m$ in ${\text {\bf U}}^{\{ ij \}}$ indicates allocation of $C_i$ to the $l^{th}$ component and $C_j$ to the $m^{th}$ component.  The calculation is similar to that described in Section \ref{sec:correlation}, but here summing over submatrices of elements of  ${\text {\bf U}}^{\{ ij \}}$ according to the partition to obtain expressions analogous to the $p_{11}$, $p_{21}$, $p_{12}$, and $p_{22}$ terms in that section. So define
\begin{align}
u_{lm}= &f(C_i=l, C_j=m \,|\, {\pmb C}_{\backslash \{ i,j \}} ,  \beta,  \, \theta_0, {\pmb Y}  ) \nonumber \\
\propto &  \prod_{k=1}^K  \frac{\;\; \Gamma (\beta + n^*_k )   \;\; \Pi_{i=1}^D  \Gamma (\nu_0 + n^*_{k}+1-i) }{ |S^*_{\{k\}}|^{\nu_0/2 + n^*_k/2} \; (\kappa_0 + n^*_k)^{D/2} } 
\label{eqn:product}
\end{align}
where 


\begin{align*}
|& S^{*}_{\{k\}}| =   \Bigr| S_{{\{k \backslash b \}}}  + I_{l=m}V_{\{ ij \}}+  \;\;\;\;\;\;\;\;\;\;\;\;\;\;\;\;\;    \\  & \frac{(I_{k=l} + I_{k=m} )  \cdot n_{k\backslash b}}{n_{k\backslash b}+I_{k=l} + I_{k=m}   } (\overline{Y}_{\{ b_k \}_{lm}} - \overline{Y}_{\{k\backslash b \}} )^T(\overline{Y}_{\{ b_k \}_{lm}} - \overline{Y}_{\{k\backslash b \}} ) \Bigr|  
\end{align*}

$$ n_k^{{\tiny *}}=   n_{k\backslash b} +I_{k=l} + I_{k=m} $$
and

$$\overline{Y}_{\{ b_k \}_{lm}} = \frac{ \Bigl( I_{k=l} Y_i + I_{k=m} Y_j    \Bigr)}{I_{k=l} + I_{k=m}} $$
where $V_{\{ ij \}}$ is the variance of the $Y_i$ and $Y_j$, and 
 $I_{l=m}$ is a (condensed) indicator function for $C_i$ and $C_j$'s allocations to the same component (ie, $C_i=l=m=C_j$) and is 0 otherwise.
 Also, one observes that with the addition of the indicators to $n_{k\backslash b}$ for the $C_i$ and $C_j$ joint distribution calculation, the $n_k^*$ for one set of indices $l,m$ may be a different value than the $n_k^*$ from another set of indices. And while in the $K=2$ case each $p_{lm}$ consisted of the product of 2 terms, here each $u_{lm}$
  is a $K$ term product.







The $p_{11}^{k'}$, $p_{12}^{k'}$, $p_{21}^{k'}$, and $p_{22}^{k'}$ terms needed for calculating correlation of $I(C_i \leq k')$ and $I(C_j \leq k')$ are $$p_{11}^{k'} = \gamma \, \sum_{l\leq k', \\ m\leq k'} u_{lm} \,\, , \;\;\;\;\; p_{22}^{k'} = \gamma \, \sum_{l> k', \\ m> k'} u_{lm} $$
$$p_{21}^{k'} = \gamma \, \sum_{l> k' \\m\leq k'} u_{lm} \,\,  , \;\;\;\;\;\;\mbox{and}\;\;\;\;\;\; p_{12}^{k'} = \gamma \, \sum_{l\leq k', \\m> k'} u_{lm} $$
where we recycle $\gamma$ as a normalizing constant.  The lower bound of the convergence rate $\rho_{\tilde{\pi}} (  {\pmb C}_{\{i,j\}})$ is then
\begin{align}
\bigl((p_{11}^{k'} - p_{1\cdot}^{k'} p_{\cdot 1}^{k'}  )/((p_{11}^{k'} + p_{12}^{k'})(p_{21}^{k'} + p_{22}^{k'}))  \bigr)^2 = Cor(C_i^{k'},C_j^{k'})^2\leq \rho_{\tilde{\pi}} ({\pmb C}_{\{i,j\}})
\label{eqn:pi}
\end{align}
using symmetry in $p_{12}^{k'}$ and $p_{21}^{k'}$. 
As in the $K=2$ case, the diagonal entries of the contingency table become arbitrarily large relative to the off-diagonal as the distance between the means of the blocking set allocated to component $k$ and that component become large.  As a result, this lower bound can again become arbitrarily close to 1, which is shown more generally in Theorem \ref{theorem_main}.

While any function that is an indicator for a subset of component allocations will provide a valid lower bound on the convergence rate of $ \rho_{\tilde{\pi}} ({\pmb C}_{\{i,j\}})$, tighter bounds can be achieved when $p_{11}^{k'}$ and $p_{22}^{k'}$ are closer in value.



\subsection{$K>2$, $B>2$ case}
For $B >  2$, we calculate the joint distribution of $C_i$ and $C_j$ having marginalized out ${\pmb C}_{\{b \}\backslash ij}$, the blocked latent variables excluding $C_i$ and $C_j$.  Then one can similarly partition that joint distribution according to an indicator function as in the $K>2$, $B=2$ case to calculate a correlation, which one can show
will be a lower bound on the convergence rate of ${\pmb C}_{\{b \}}$ to ${\pmb {\tilde \pi}}^{C_{\{b \}}}$ under Gibbs sampling.  Since calculation of the marginalized joint distribution requires a sum over all allocations of the blocking set complementary to $C_i$ and $C_j$, computation becomes more difficult the larger $B$.  We calculate it with the expression
\begin{align}
\sum_{{\pmb C}_{\{b \}  \backslash ij} \in {\mathcal C_{ \{b \}  \backslash ij}}} f(C_i=l,C_j=m , {\pmb C}_{\{ b \} \backslash ij }   | & {\pmb C}_{\backslash \{ b \}}=c_{\backslash \{ b \}},  \beta,  \, \theta_0, Y ) \nonumber \\
=f(C_i=k_1,C_j=k_2 & |  {\pmb C}_{\backslash \{ b \}}=c_{\backslash \{ b \}} ,  \beta,  \, \theta_0, Y ) 
\label{eqn:integrate}
\end{align}
where ${\mathcal C_{ \{b \}  \backslash ij}}$ is the collection of possible allocations of ${\pmb Y}_{\{ b\} \backslash ij}$ to the $K$ components.  Under the assumption of identical valued $Y_i$ for $i\in \{b \}$, the sum is not overly burdensome with $K+B-2 \choose B-2$ possible allocations in contrast to the $K^{B-2}$ allocations otherwise.  But by designation of outliers on a common set of components, observations will tend to be close to one another in value, especially with respect to their contribution to sums of squares of the components into which they could be sampled, and so the assumption is not especially restrictive.  Since there are only $B-2$ different possible sums of squares for a given component and on $K$ components, one can cache these values and then combine as needed according to each $K+B-2 \choose B-2$ allocation for efficient computation.  
The sum can take the form of constructing a $K \times K$ contingency table representing the joint distribution of $C_i$ and $C_j$ for each of the $K+B-2 \choose B-2$ allocations of the complementary set $C_{\{b \}\backslash ij}$ and summing across the tables element wise, yielding an unnormalized $K\times K$ table corresponding to the bivariate joint distribution defined in Equation (\ref{eqn:integrate}).  



To perform the calculation, for each ${{\pmb C}_{\{b \}\backslash ij}}\in  {\mathcal C_{ \{b \}  \backslash ij}}$, define  
${\textbf U}_{C_{\{b \}\backslash ij}}$ as the $K\times K$ matrix giving the joint distribution of $C_i, C_j$ under allocation ${\pmb C}_{\{b \}\backslash ij}$ of the other $B-2$ observations ${\pmb Y}_{\{b \}\backslash ij}$ in the blocking set.  That is, for the $l,m$ element of ${\textbf U}_{C_{\{b \}\backslash ij}}$ which we denote $[  {\textbf U}_{C_{\{b \}\backslash ij}} ]_{lm}$ we have


\begin{align}
[  {\textbf U}_{C_{\{b \}\backslash ij}} ]_{lm} \propto  f(C_i=l, C_j=m,{\pmb C}_{\{b\}\backslash ij} | {\pmb C}_{\backslash \{ b \}} ,  \beta,  \, \theta_0, Y  ) \nonumber \\ 
\propto \prod_{k=1}^K  \frac{\;\; \Gamma (\beta + n^{**}_k )   \;\; \Pi_{i=1}^D  \Gamma (\nu_0 + n^{**}_{k}+1-i) }{ |S^{**}_{\{k\}}|^{\nu_0/2 + n^{**}_k/2} \; (\kappa_0 + n^{**}_k)^{D/2} } 
\label{eqn:product2}
\end{align}
where 

\begin{align*}
|& S^{**}_{\{k\}}| =   \Bigr| S_{{\{k \backslash b \}}}  + V^{C_{\{b \}\backslash ij}}_{\{ b_k \}_{lm}} +  \;\;\;\;\;\;\;\;\;\;\;\;\;\;\;\;\;    \\  & \frac{(I_{k=l} + I_{k=m} +n^{C_{\{b \}\backslash ij}}_k )  \cdot n_{k\backslash b}}{n_{k\backslash b}+I_{k=l} + I_{k=m}    +n^{C_{\{b \}\backslash ij}}_k  } (\overline{Y}^{C_{\{b \}\backslash ij}}_{\{ b_k \}_{lm}} - \overline{Y}_{\{k\backslash b \}} )^T(\overline{Y}^{C_{\{b \}\backslash ij}}_{\{ b_k \}_{lm}} - \overline{Y}_{\{k\backslash b \}} ) \Bigr|  
\end{align*}

$$ n_k^{{\tiny **}}=   n_{k\backslash b} +I_{k=l} + I_{k=m} + n^{C_{\{b \}\backslash ij}}_k$$
and

$$\overline{Y}^{C_{\{b \}\backslash ij}}_{\{ b_k \}_{lm}} = \frac{\Bigl( I_{k=l} Y_i + I_{k=m} Y_j  + \sum_{\substack{ \{n: C_n=k\; \tiny{} \\ \mbox{\tiny{under}} \;C_{\{ b\}\backslash ij }\}}} Y_n   \Bigr)}{I_{k=l} + I_{k=m} + n^{C_{\{b \}\backslash ij}}_k}  $$
that is, the $Y_n$'s allocated to component $k$ under ${\pmb C}_{\{ b\}\backslash ij}$ with the addition of $Y_i$ and $Y_j$ per the $l,m$ indexing, indicative of the allocations of $C_i$ and $C_j$ to components $l$ and $m$, respectively, for calculating their joint distribution, while
$V^{C_{\{b \}\backslash ij}}_{\{ b_k \}_{lm}}$ is the variance of the $Y_n$'s involved in the mean calculation of $\overline{Y}^{C_{\{b \}\backslash ij}}_{\{ b_k \}_{lm}} $.

 While the notation becomes heavy,
 the important point is that the calculations are driven by the difference between the component mean without the blocked observations and the mean of the observations in the block allocated to that component, which may include $Y_i$ and $Y_j$ as the joint distribution of $C_i$ and $C_j$  stipulates.

Now define $\textbf{U}^{\{ b \}}$ to be the $K\times K$ matrix giving the joint distribution of $C_i, C_j$ having marginalized out the complementary set of latent variables in the blocking set, ${\pmb C}_{\{ b\}\backslash ij}$, that is, $\textbf{U}^{\{ b \}}$ is 
$$ \textbf{U}^{\{ b \}}  = \sum_{{C_{\{b \}\backslash ij}} \in {\mathcal{C}_{\{b \}\backslash ij}}} 
      {\textbf U}_{C_{\{b \}\backslash ij}} 
$$
Then we claim the following:

\begin{theorem}
For $u^{\{b \}}_{lm}$, the $l,m$ element of $\textbf{U}^{\{ b \}}$, $\gamma$ a normalizing constant, and for probabilities $p_{11}^{\{b \}} = \gamma \, \sum_{l\leq k', \\m\leq k'} u_{lm}^{\{b \}}, p_{22}^{\{b \}} = \gamma \, \sum_{l> k', \\m> k'} u_{lm}^{\{b \}}, p_{21}^{\{b \}} = \gamma \, \sum_{l> k', \\ m\leq k'} u_{lm}^{\{b \}} ,$ and $p_{12}^{\{b \}} = \gamma \, \sum_{l\leq k', \\m> k'} u_{lm}^{\{b \}} $, 
then the lower bound for  $\rho_{\tilde{\pi}} ({\pmb C}_{\{b \}})$ becomes 
\begin{align*}
\bigl((p_{11}^{\{b \}} - p_{1\cdot}^{\{b \}}p_{\cdot 1}^{\{b \}}  )/((p_{11}^{\{b \}} + p_{12}^{\{b \}})(p_{21}^{\{b \}} + p_{22}^{\{b \}})) \bigr)^2 \leq  \rho_{\tilde{\pi}} ({\pmb C}_{\{i,j\}}) \leq  \rho_{\tilde{\pi}} ({\pmb C}_{\{b \}}) 
\end{align*}
\label{theorem_general_KB}
\end{theorem}
%
%
\noindent
\begin{proof}
Using the argument in the Proof of Lemma \ref{lemma_invariant} with respect to constructing self-adjoint operators for reversible chains and inferring on their spectral norms, one can apply Liu's Theorem 1, concluding that marginalization reduces the convergence rate and giving the second inequality \citep{liu_collapsed_1994}.
The first inequality follows from Lemma \ref{lemma_multi_K}.


\end{proof}


\begin{theorem}
Consider the scenario from Theorem \ref{theorem_cor} 
where $d_k$ scales $\overline{Y}_{\{k \backslash b \}}$ with $\overline{Y}_{\{k \backslash b \}}\neq 0$ for all $k \ \in \{1, \dots , K \}$, with $d_k \rightarrow \infty$ so that the distance from 
$\overline{Y}^{C_{\{b \}\backslash ij}}_{\{ b_k \}}$ to component mean $d_k \, \overline{Y}_{\{k \backslash b \}}$ increases without bound for all $C_{\{b \}\backslash ij} \in {\mathcal{C}_{\{b \}\backslash ij}}$, where $d_l = O(d_m)$ for all $l,m$.  
Assume $n_{k\backslash b}>0$ for all $k$.  Then for $I(\cdot \leq k')$ for any $1\leq k' < K$, $$(p_{11}^{\{b \}} - p_{1\cdot}^{\{b \}}p_{\cdot 1}^{\{b \}}  )/\bigl((p_{11}^{\{b \}} + p_{12}^{\{b \}})(p_{21}^{\{b \}} + p_{22}^{\{b \}}) \bigr)  \rightarrow 1 $$  
Since  $(p_{11}^{\{b \}} - p_{1\cdot}^{\{b \}}p_{\cdot 1}^{\{b \}}  )/\bigl((p_{11}^{\{b \}} + p_{12}^{\{b \}})(p_{21}^{\{b \}} + p_{22}^{\{b \}}) \bigr) \leq \rho_{\tilde{\pi}}({\pmb C}_{\{b \}})$, we have 
$ \rho_{\tilde{\pi}}({\pmb C}_{\{b \}}) \rightarrow 1 $
\label{theorem_main}
\end{theorem}

Since the convergence rate of ${\pmb C}_{\{b\}}$ to $\tilde{\pi}^{C_{\{b \}}}$ is not bounded away from $1$, it puts into question the clustering characterization of these observations under sampling of any length.  
If lower bounds on the convergence rate are high, it may often be desirable to use a sampling scheme that converges faster to ${\bf {\tilde{\pi}}}^{C_{\{b \}}}$.  Here we propose a blocked Gibbs sampling procedure that improves the convergence rate and balances trade-offs between computational burden and convergence benefit.

\section{Blocked Gibbs sampling }
\label{sec:blocking}


A blocked Gibbs sampling approach to this problem identifies those groups of latent variables that exhibit significant correlation in the parameter space, in this case those observations close to one another and in the tails between component distributions, and samples them as a block.  So during the Gibbs sampling, we sample from this group jointly, conditional on the complementary set of latent variables, that is ${\pmb C}_{\{ b\}} | {\pmb C}_{\backslash \{ b\}}$, and then proceeds with standard Gibbs sampling, $C_{i} | {\pmb C}_{\backslash i}$ for those $i \notin \{b \}$.  Considering the partitioning of Equation (\ref{eqn:partition_space}), the small cost we pay in ${\pmb {\tilde{\pi}}}\neq {\pmb {\pi}}$, differing by a small $\epsilon$, is compensated for in convergence to ${\pmb {\tilde{\pi}}}$ faster than we could to ${\pmb \pi}$.

Using this procedure, $\rho_{\tilde{\pi}} ({\pmb C}_{\{ b \}})=0$ because we are sampling from what is ${\pmb C}_{\{b \}}$'s ``marginal'' distribution conditional on the assumed high probability allocation of the complementary set.   While the modified Gibbs sampling procedure is beneficial from the perspective of convergence speed, the trade-off is computational burden: for each sample from ${\pmb C}_{\{ b\}} | {\pmb C}_{\{\backslash b\}}$, the probability for $B^K$ possible moves must be calculated, in contrast to the $B\cdot K$ calculations total that are needed for $B$ consecutive moves under the standard Gibbs sampler.
In addition, while standard Gibbs sampling can use rank-one updates to $S_{\{k \backslash i\}}$ in calculating $S_{\{k \}}$ to ease the computational burden of determinant calculation, for $S_{\{k \backslash b\}}$ no such shortcut can be assumed in calculating the joint distribution for  ${\pmb C}_{\{b \}}$.  

A modification of the blocked Gibbs sampler must therefore be used to address these practical computational challenges.  To calculate the $K^B$ possible allocations of ${\pmb C}_{\{ b\}}$, we use a second-order approximation of the determinant, which is necessary because a first order approximation performs poorly in exactly those situations where it is needed--the presence of outliers significantly perturbs component sums of squares calculated under possible latent variable allocations.  After sampling from one of these allocations according to the approximation, we calculate an exact value for the proposed move and perform an accept-reject step as a function of their ratio so detailed balance holds.  The procedure can be summarized as follows:


\begin{enumerate}
\item[Step 1] Calculate $\hat{f}({\pmb C}^{(n+1)}_{\{b\}} | {\pmb C}^{(n)}_{\backslash \{ b \}} ,  \beta,  \, \theta_0, {\pmb Y})$, an approximation of $f({\pmb C}^{(n+1)}_{\{b\}} | {\pmb C}^{(n)}_{\backslash \{ b \}} ,  \beta,  \, \theta_0, {\pmb Y})$, the condition distribution of ${\pmb C}^{(n+1)}_{\{ b\}} $ given $ {\pmb C}^{(n)}_{\backslash \{ b\}}$ using a second order determinant approximation on the $K^B$ allocations for ${\pmb C}_{\{ b\}}$, that is, for a cached $|S_{\{ k\backslash b \}}|$ and perturbation $\epsilon \, Q_k$,  
$$ | S_{\{ k\backslash b \}} + \epsilon \, Q_k | = |S_{\{ k\backslash b \}}| \Bigl(1 + \epsilon \; tr(A_k) + \epsilon^2/2\cdot \bigl(tr(A_k)^2 - tr(A_k^2)\bigr) \Bigr)$$
where $A_k = Q_k S_{\{ k\backslash b \}}^{-1}$


\item[Step 2] Sample ${\pmb C}^{(n+1)}_{\{ b\}}$, a proposed move, from the multinomial distribution according to $\hat{f}({\pmb C}^{(n+1)}_{\{b\}} | {\pmb C}^{(n)}_{\backslash \{ b \}} ,  \beta,  \, \theta_0, {\pmb Y})$

\item[Step 3] Calculate the true value, $f({\pmb C}^{(n+1)}_{\{b\}} | {\pmb C}^{(n)}_{\backslash \{ b \}} ,  \beta,  \, \theta_0, {\pmb Y})$, and
define the ratio $$r_{C^*_{\{b \}}} = \frac{f({\pmb C}^{(n+1)}_{\{b\}} | {\pmb C}^{(n)}_{\backslash \{ b \}} ,  \beta,  \, \theta_0, {\pmb Y})}{\hat{f}({\pmb C}^{(n+1)}_{\{b\}} | {\pmb C}^{(n)}_{\backslash \{ b \}} ,  \beta,  \, \theta_0, {\pmb Y})}$$

\item[Step 4] Draw $u \sim \mbox{Unif}(0,1)$, the uniform distribution on $(0,1)$, and if $r_{C^*_{\{b \}}}>u$, accept the move ${\pmb C}^{(n)}_{\{b \}}\rightarrow {\pmb C}^{(n+1)}_{\{b \}}$, and otherwise return to Step 1

\end{enumerate}

\noindent
If it is not feasible to block sample all the $C_i$ in the outlier set, one can instead block smaller subsets such that their union comprise ${\pmb C}_{\{b \}}$.  While in that case it is difficult to provide convergence guarantees on the procedure, the rate to $\tilde{{\pmb \pi}}$ will necessarily improve as compared to non-blocked sampling by \citet{liu_collapsed_1994} Theorem 1.






\section{Synergism, partial Synergism, and Antagonism in the latent variable space}
\label{sec:synergism}
We have focused in the discussion above on the case where outliers are in some neighborhood of one another.  In this case, there is correlation of latent variables associated with all components into which the observations can be sampled.  One might call the phenomenon synergism.  One can consider that this behavior could also apply only to subsets of a mixture model's components depending on the observations considered, or partial synergism.  That is, latent variables could be correlated with respect to being sampled into one component, but they are independent or even antagonistic with respect to a complementary set of components.

In notation, we have been considering the case where for any $i,j\in \{ b\}$, $Cor(I(C_i=k),I(C_j=k))\geq 0$ on a common set of $k$ for which there is non-trivial probability for either $C_i$ or $C_j$ to be sampled into that component.  Whereas, one may consider alternatives where for some $k$, $Cor(I(C_i=k),I(C_j=k))\approx 0$ while $P(C_i=k)>0$ or $P(C_j=k)>0$, or also that possibility that $Cor(I(C_i=k),I(C_j=k))< 0$.  This synergistic or antagonistic behavior becomes a function of the spatial orientation of components (themselves a function of allocations of observations) with respect to subsets of outlier observations which may not be in proximity of one another.  One might consider two such orientations of two observations and three components in Figures \ref{fig:antag} and \ref{fig:partial}, one exhibiting antagonistic behavior with respect to latent variable allocation to component 2 (Figure \ref{fig:antag}), and one exhibiting partial synergistic behavior with respect to latent variable allocation to cluster 2 (Figure \ref{fig:partial}).  In either case, allocation of the two observations to clusters 1 and 3 separately would be nearly independent.  

\begin{figure}[H]
\begin{center}
\includegraphics[width=0.7\textwidth]{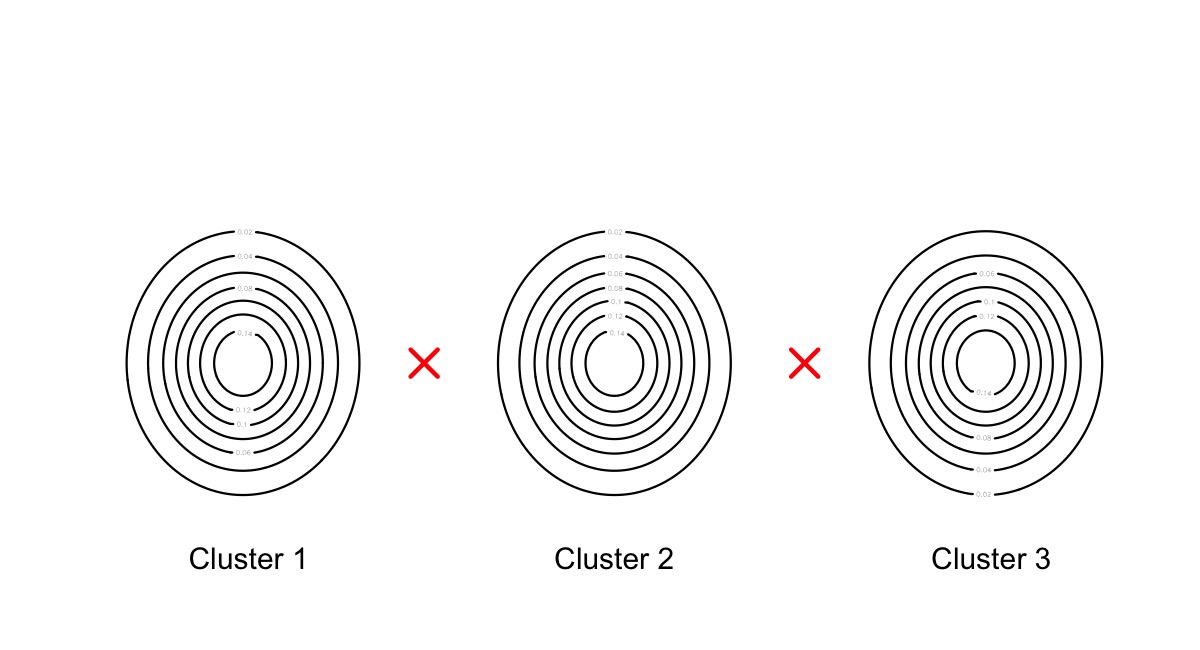}
\captionsetup{width=0.9\textwidth}
\caption{{\small Contour plot of a mixture density composed of three Gaussian clusters or components, with an observation each between the two pair of adjacent components, marked with red X's.  The latent variables of these observations would negatively correlate with respect to cluster 2 because they ``pull'' in different directions, but behave independently with respect to allocation to clusters 1 or 3, respectively.}}
\label{fig:antag}  
\end{center}
\end{figure}

\begin{figure}[H]
\begin{center}
\includegraphics[width=0.7\textwidth]{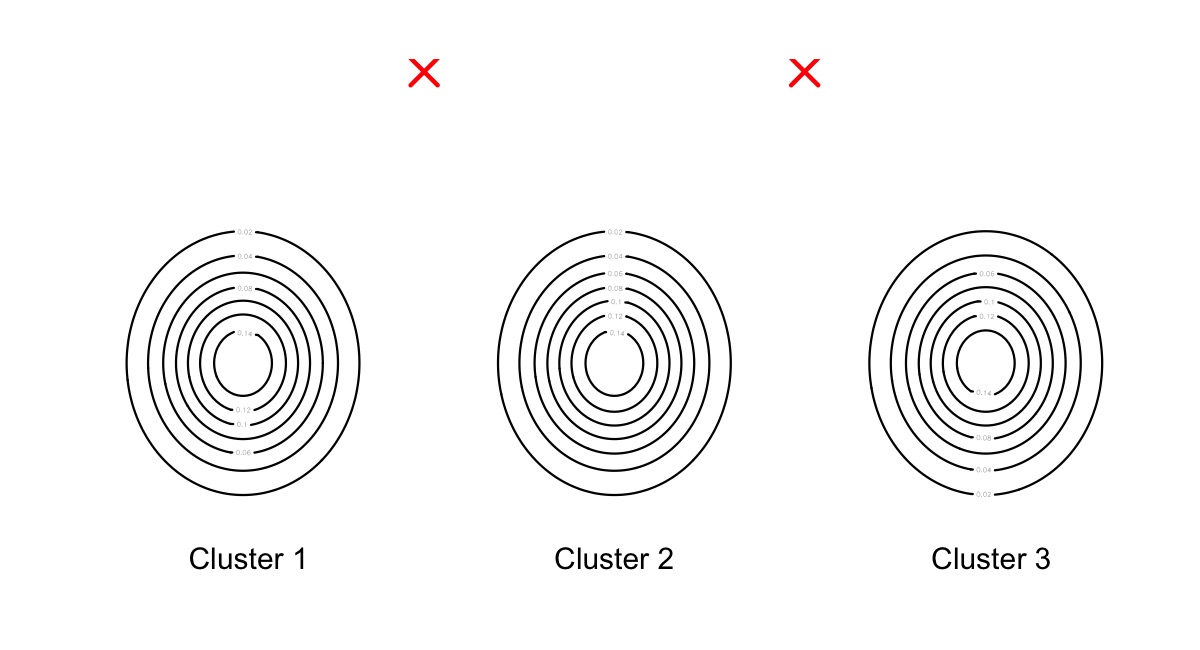}
\captionsetup{width=0.9\textwidth}
\caption{{\small Contour plot of a mixture density composed of three Gaussian clusters or components, with an observation each midway and above the two pair of adjacent components, marked with red X's.  The latent variables of these observations would correlate with respect to cluster 2, but behave independently with respect to allocation to clusters 1 or 3, respectively. }}
\label{fig:partial}  
\end{center}
\end{figure}

In these cases, one must again rely on notions of some $(1-\epsilon)$ size probability region of the posterior space of latent variables one is primarily interested in
or likewise clipped tails of component distributions.  Otherwise, correlation will be still more non-trivial for distant components since observations are only increasingly in their tails and are therefore well-leveraged.  

For this reason, studying notions of convergence in these settings even ignoring challenges with label switching
becomes difficult.
We do not develop strategies for improving sampling in these cases, but introduce and distinguish these dynamics from those we have focused on in this study.  We remark that the case or spatial orientation where observations are in tails of components and in a neighborhood of one another uniquely exhibits positive correlation of latent variables on a common set of components and so renders the case of strong interest to understand.  We examine other orientations in another manuscript.








\section{Simulation}
\label{sec:simulation}

We generated a mixture distribution of 160 observations with four components, each Gaussian in three dimensions, whose means were all equidistant from one another.  We placed three outlier observations equally between two or three of the four component means depending on the simulation and proceeded with Gibbs sampling the data using a four component collapsed Gaussian mixture model with non-informative hyperparameters: $S_0$, a rank 3 diagonal matrix of 2, $m_0$, a vector of 0's, $\kappa_0 = 0.005$, $\nu_0 = 0.02$, and $\beta =3$.  
The covariance matrices within components were diagonal with variances of 1.  The distance between the four component means was 11 
so that clustering was clearly present and identifiable.   

Since the focus of the study was comparison of blocked versus non-blocked Gibbs sampling, in one case we block sampled with size 3 while in the other case we sampled one observation at a time.
We sampled the outlier observations at regular intervals to assure a sufficient number of samples to assess performance of the sampler.  For the blocked sampling case, the interval was once every 20 observations (each of those 20 iterations consisting of a block sample of 3), and for the standard Gibbs sampler, every 60 (each iteration consisting of one new sample), at which point each of the three outliers was chosen consecutively.  In this way we achieved comparability of the two approaches since an identical number of latent variables were sampled between interval.  A total of 15000 iterations were computed for the blocked sampler, and 45000 for the standard Gibbs sampler so that the total number of samples was equal.  The last 4000 samples were examined for their clustering patterns, where those 4000 consisted of ``thinned'' samples in the standard Gibbs sampler case so that one of every three was taken.  In this way, the total number of elapsed moves was equivalent between the blocked and non-blocked regimes.  We fit posterior similarity matrices (PSMs) \citep{fritsch_improved_2009}, which map the estimated posterior probability that $P(C_i=C_j)$ to a color for every element $1\leq i,j\leq N$, and calculated autocorrelations from these chains.  The autocorrelation function (ACF) was fit on one sample of the outlier set which, because it only vacillated between the two components, could be considered binary.  

The PSMs for the scenario of outliers equidistant between two components for the blocked versus non-blocked alternatives are shown in Figures \ref{fig:psm2}a (non-blocked sampling) and \ref{fig:psm2}b (blocked sampling).  Both PSMs exhibit clearly identifiable clustering with respect to the four components indicating the samplers generally perform as expected.  One also observes the mutual clustering of the three outliers at the bottom of the PSMs, which in both figures cluster together as one might expect.  However, the clustering shown in Figure \ref{fig:psm2}b is more consistent with what intuition and the data would suggest--for these outliers equidistant between two components, their cluster membership should be equally distributed between those components, which is what we observe.  In contrast, in Figure \ref{fig:psm2}a, the three outliers are stuck, together, in one of the components to which they are closest, having significantly pulled its mean towards themselves because they are well-leveraged, making ``escape'' difficult one by one in the absence of a blocked move.  Their allocation into that component of the two is arbitrary; if one were to start the chain again, they would as likely become allocated to the other one.

\begin{figure}[H]
\centering
\begin{subfigure}{0.48\textwidth}
\includegraphics[width=\linewidth]{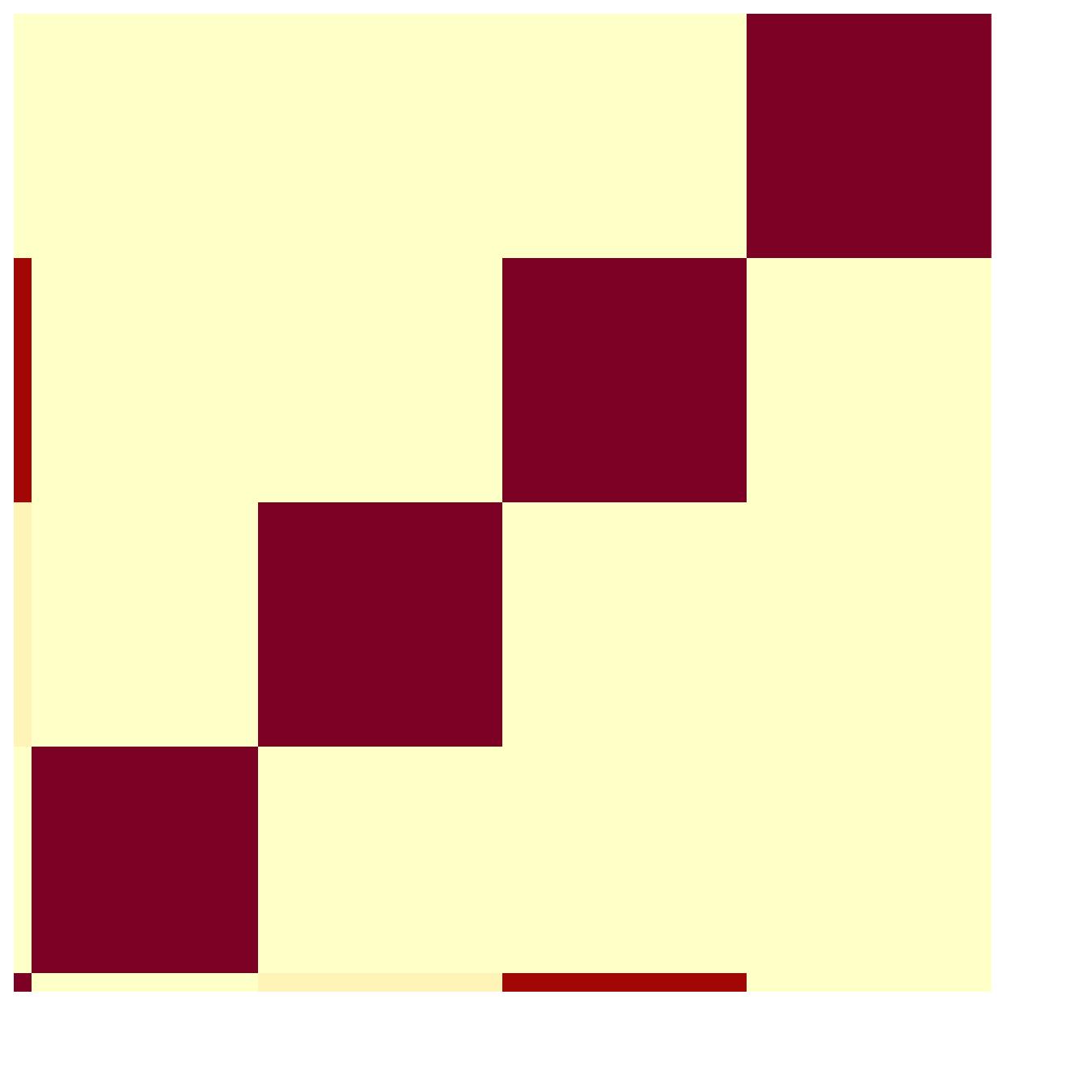}
\captionsetup{width=0.9\textwidth}
\caption{Non-blocked sampling PSM} 
\label{fig:diff1}
\end{subfigure}
\hspace*{0.25cm} 
\begin{subfigure}{0.48\textwidth}
\includegraphics[width=\linewidth]{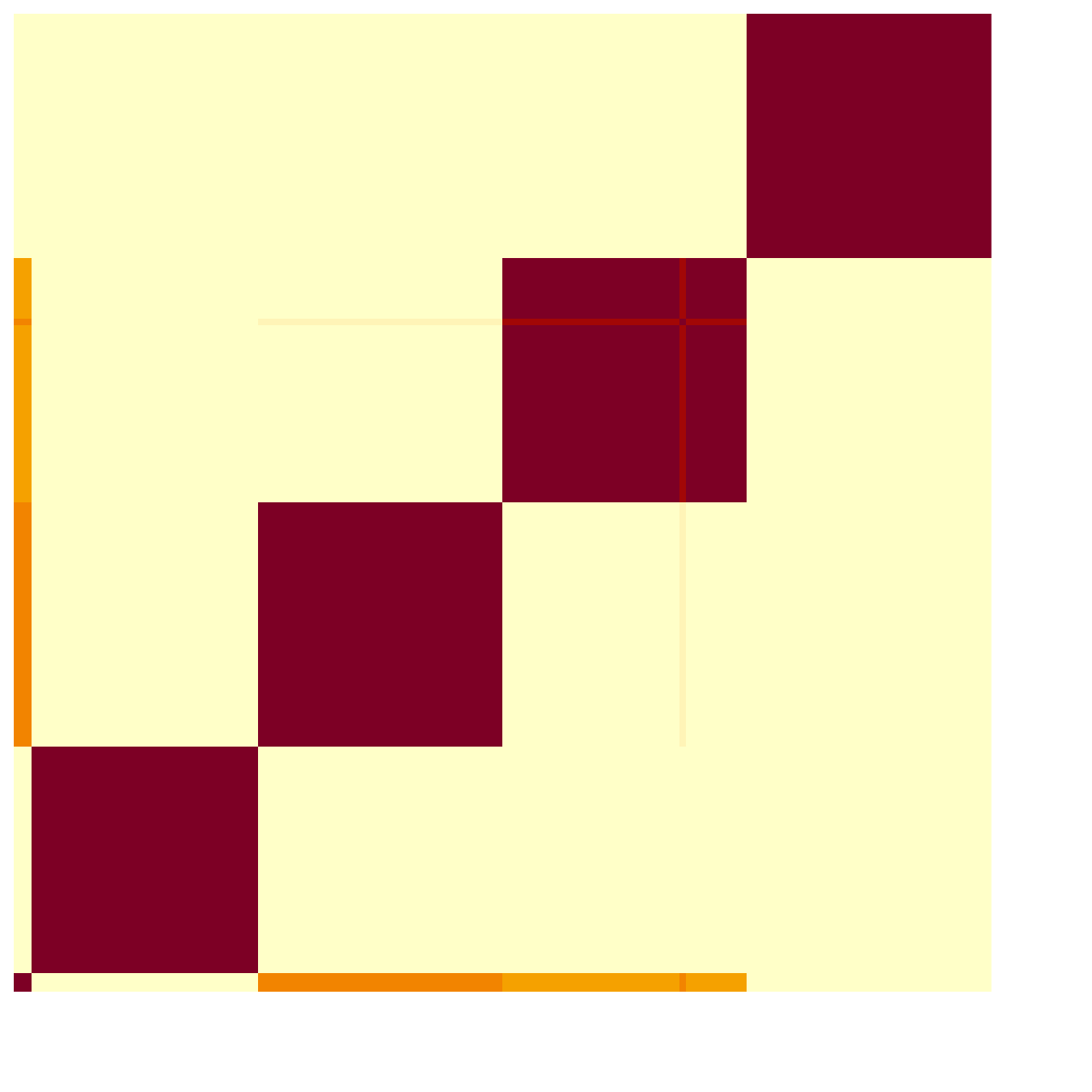}
\caption{Blocked sampling PSM} 
\label{fig:diff2}
\end{subfigure}
\captionsetup{width=0.90\textwidth}
\caption{{\small Posterior similarity matrices (PSMs) under the two different sampling regimes when the outlier block is equidistant between two components.  Observations are in the same order along the x- and y-axes, where red indicates higher values of estimated $P(C_i=C_j)$ for PSM entry $i,j$ in posterior samples, and yellow indicates lower values.}}
\label{fig:psm2}
\end{figure}
The estimated ACFs for the non-blocked (Figure \ref{fig:acf}a) and blocked (Figure \ref{fig:acf}b) regimes are consistent with clustering patterns of the PSMs--the high autocorrelation in the non-blocked sampler is consistent with outlier observations unable to escape their allocation to a component.  Conversely, the low autocorrelation in the blocked sampler is consistent with outliers' nearly equal distributions in the PSM between the two components to which they are closest.  Lastly, in Figures \ref{fig:psm3}a (non-blocked sampling) and \ref{fig:psm3}b (blocked sampling) we show PSMs for the setting where outliers are equidistant between three rather than two  components.  The patters are consistent with Figure \ref{fig:psm2}--for non-blocked sampling, outliers become stuck to one component, whereas for blocked sampling, the outliers are equi-distributed between the three components, as is more representative of the true nature of the data.






\begin{figure}[H]
\centering
\begin{subfigure}{0.48\textwidth}
\includegraphics[width=\linewidth]{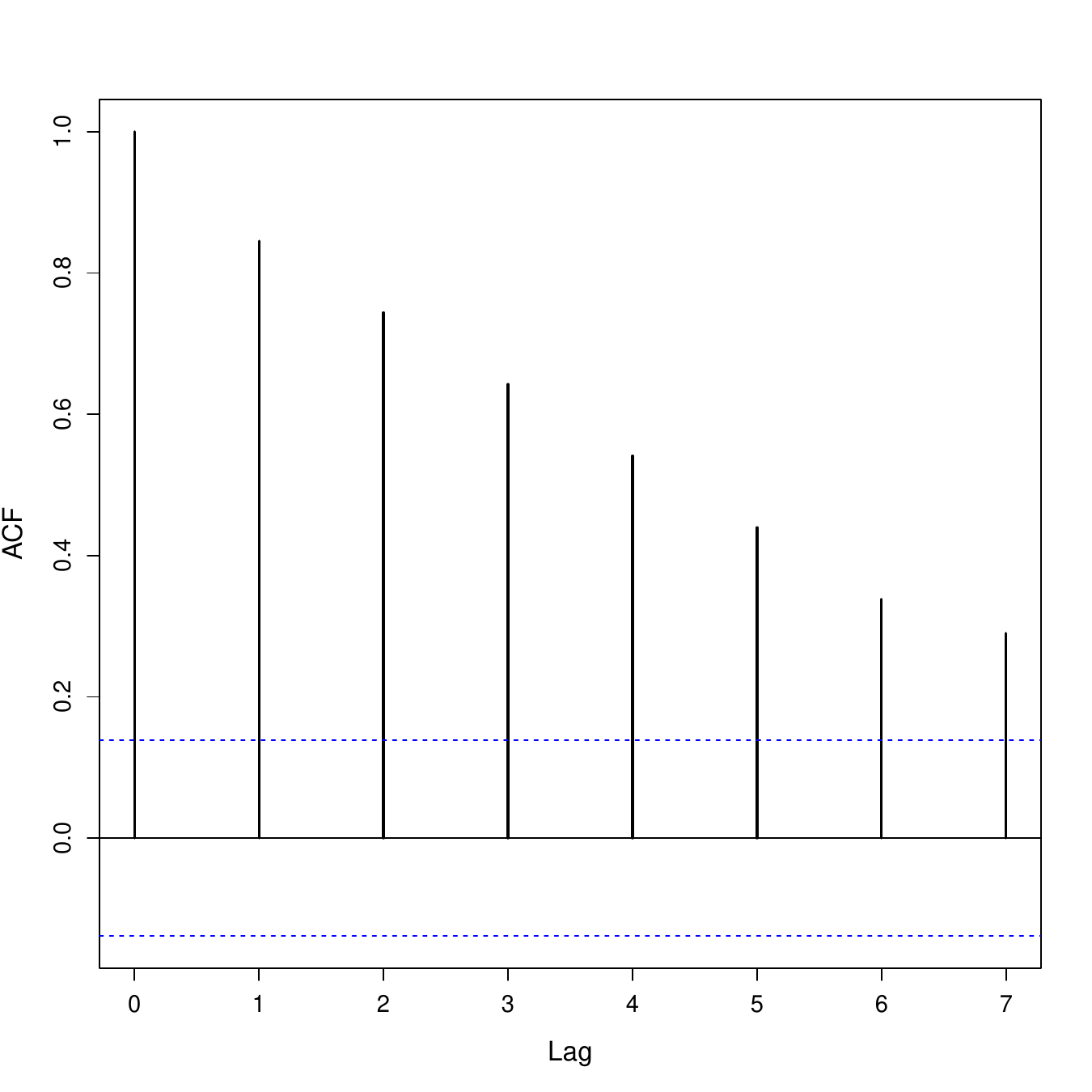}
\captionsetup{width=0.9\textwidth}
\caption{Non-blocked sampling ACF} 
\label{fig:diff1}
\end{subfigure}
\hspace*{0.25cm} 
\begin{subfigure}{0.48\textwidth}
\includegraphics[width=\linewidth]{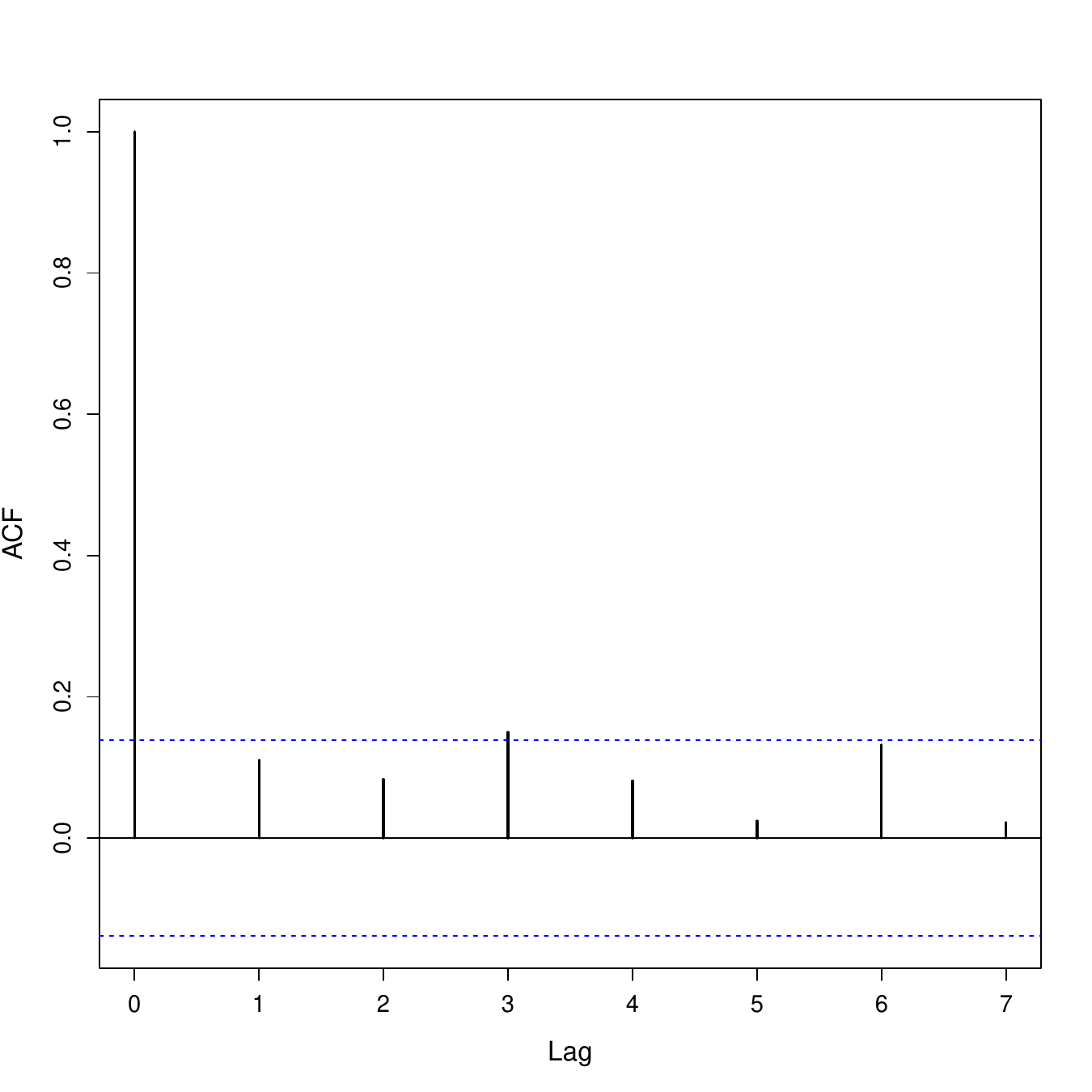}
\caption{Blocked sampling ACF} 
\label{fig:diff2}
\end{subfigure}
\captionsetup{width=0.9\textwidth}
\caption{{\small Estimated autocorrelation functions from posterior samples of a single outlier observation under the scenario where outliers are equidistant between the two components.}}
\label{fig:acf}
\end{figure}



\begin{figure}[H]
\centering
\begin{subfigure}{0.47\textwidth}
\includegraphics[width=\linewidth]{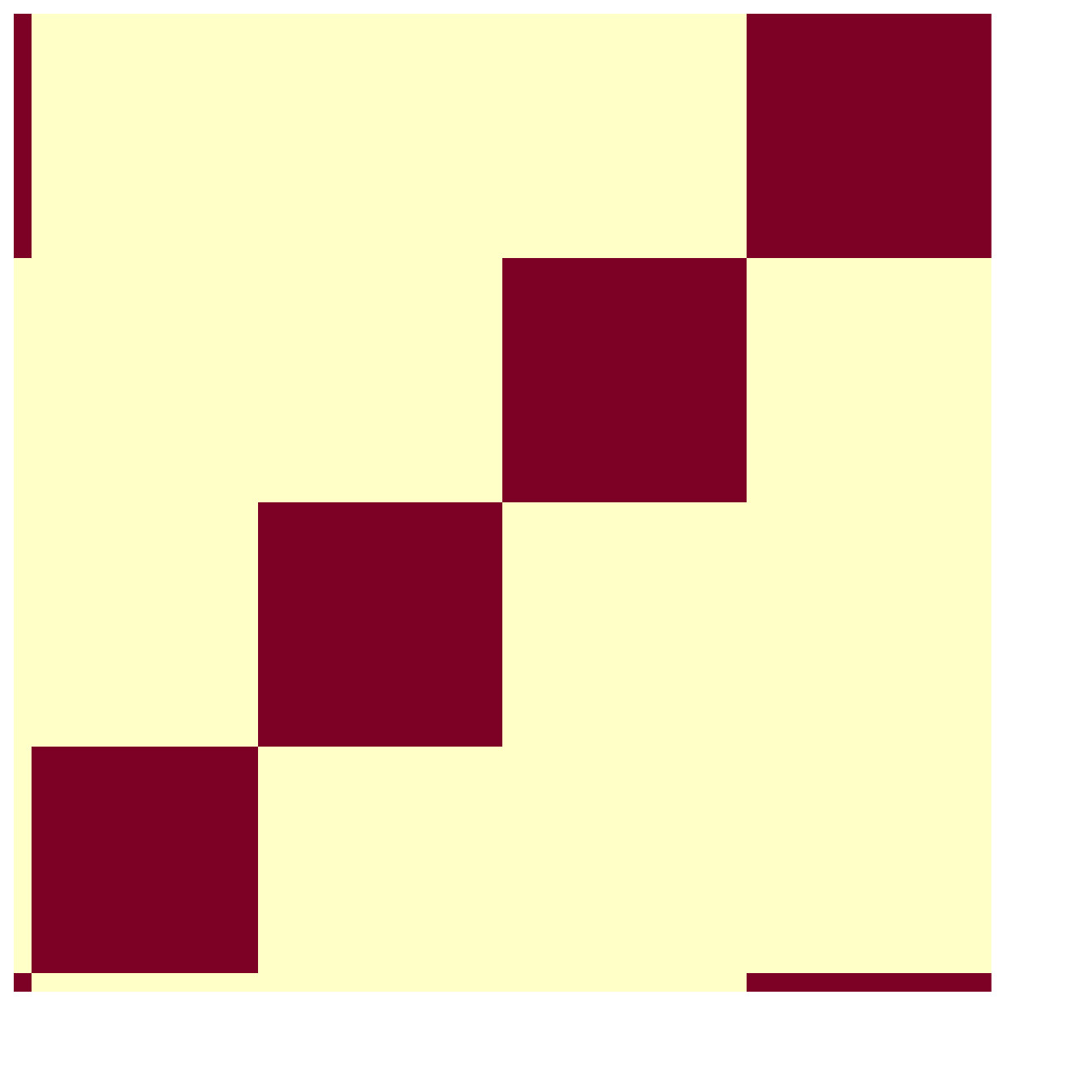}
\caption{Non-blocked sampling PSM} 
\label{fig:diff1}
\end{subfigure}
\hspace*{0.25cm} 
\begin{subfigure}{0.47\textwidth}
\includegraphics[width=\linewidth]{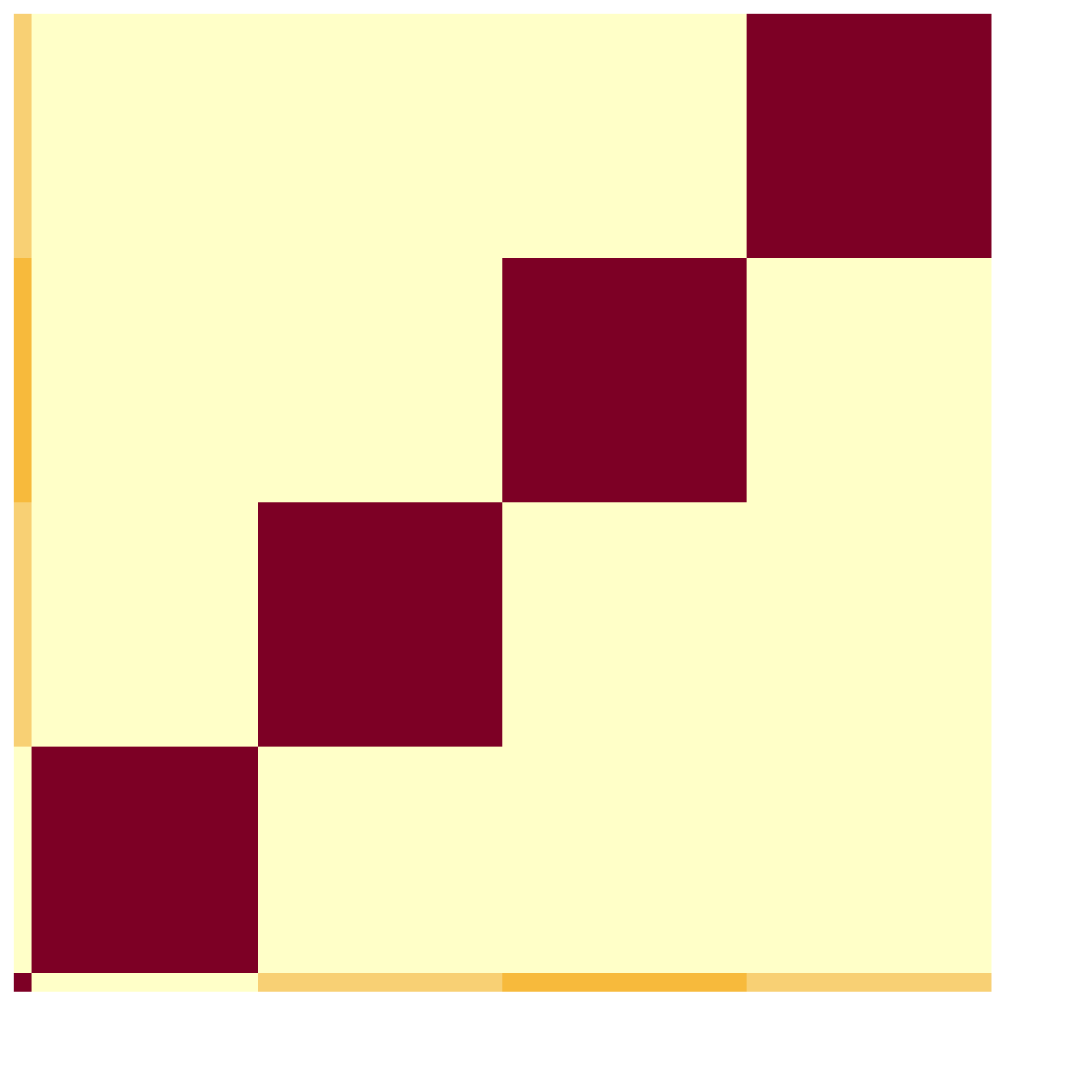}
\caption{Blocked sampling PSM} 
\label{fig:diff2}
\end{subfigure}
\captionsetup{width=0.9\textwidth}
\caption{{\small Posterior similarity matrices (PSMs) under the two different sampling regimes when the outlier block is equidistant between three components.  Observations are in the same order along the x- and y-axes, where red indicates higher values of estimated $P(C_i=C_j)$ for PSM entry $i,j$ in posterior samples, and yellow indicates lower values.}}
\label{fig:psm3}
\end{figure}











\section{Discussion}

We have described correlation structure in the space of latent variables for finite Gaussian mixture models.  That structure suggests that when observations are located approximately equidistant in the tails of component distributions, called the ``outlier set'', the respective latent variables exhibit conditional correlation that approaches 1 as the distances from those components increase in equal order.  This implies that standard, non-blocked Gibbs sampling procedures will converge arbitrarily slowly to an approximation of the outlier latent variables' stationary distributions.  That approximation was introduced because of the challenge of studying convergence to the true stationary distribution and is relevant because the approximation can become arbitrarily close to the true stationary distribution as a function of increasing clustering in the non-outlier set of the data.  From a practical perspective, a standard Gibbs sampling procedure results in outlier observations ``sticking'' in one or another component, unable to leave, so that posterior measures of clustering are not reflective of the data.  However, the characteristic is hidden by posterior diagnostics since most observations do not exhibit this behavior.  

We have shown that blocked Gibbs sampling outlier observations that reside between tails of component distributions addresses the slow convergence by ``moving correlation from the Gibbs sampler to the random number generation'' \citep{seewald_discussion_1992,roberts_updating_1997}.  The trade-off is computational complexity, and we have described a procedure that reduces computational burden by approximating expensive determinants and using an accept-reject step so that detailed balance holds.  

While it is difficult to study convergence rates with dynamic updating schemes or rigorously justify blocking rules composed of subsets of the deemed ``outlier set'' (that set a designation whose appropriateness may change to some degree as sampling proceeds and latent variable allocations are updated) due to computational difficulty, one way forward may be to first use standard, univariate Gibbs sampling to fit a finite Gaussian mixture model and investigate if clustering is evident in the data.  If there is, one can 1) identify a tentative set of outliers, 2) determine a computationally feasible block size, and then 3), since observations close to one another exhibit more correlation in the latent variable space other things held equal as exhibited in Figure \ref{fig:diff}, proceed with blocked sampling those observations in the outlier set closest to one another while retaining standard Gibbs sampling in the complementary set. 

As discussed in Section \ref{sec:synergism}, however, correlation in the latent variable space has a complicated structure, and sampling strategies which only block observations in close proximity do not encompass all scenarios where latent variables may exhibit significant correlation.  We described antagonist and partial synergistic structures, observations involved in which may also benefit from distinct blocking regimes.  

In applied data analysis settings, highly leveraged outliers between component distributions are difficult to categorize, but more or less important to understand well depending on the setting.  These observations may comprise types of cancer whose genomic patterns do not fall definitely in one or another category, and any coarse summarization of such an observation due to poor convergence in a mixture model may have negative consequences on its understanding.  While the optimality of one or another scheme is difficult to show, we have made progress in showing how blocked Gibbs sampling the outliers as described in this work will improve their convergence and therefore understanding in important applications such as these.  




\bibliographystyle{plainnat} 
\bibliography{blocked_gibbs_10_31.bib}

\clearpage

\begin{center}
\textbf{\large Supplementary Material}

\textbf{David Swanson}
\end{center}
\setcounter{equation}{1}
\setcounter{figure}{0}
\setcounter{section}{0}
\setcounter{table}{0}
\setcounter{page}{1}
\makeatletter
\renewcommand{\thetable}{S\arabic{table}}
\renewcommand{\thefigure}{S\arabic{figure}}
\renewcommand{\thesection}{S\arabic{section}}
\renewcommand{\bibnumfmt}[1]{[S#1]}
\renewcommand{\citenumfont}[1]{S#1}



\begin{proof}[{\bf Proof of Theorem \ref{theorem_cor}}]


Using invertability of $S_{{\{k \backslash b \}}} + I_{l=m}V_{\{ij\}}$ for $k=1,2$, we can write
$$\Gamma_{lm} \Bigl|S_{{\{k \backslash b \}}} + q_{lm}\cdot \frac{n_{k\backslash b}}{n_{k\backslash b}+q_{lm}} (d_k \,\overline{Y}_{\{ k \backslash b\}} -  \overline{Y}_{\{b_k\}})(d_k \, \overline{Y}_{\{ k \backslash b\}} -  \overline{Y}_{\{b_k\}})^T   + I_{l=m}V_{\{ ij \}}\Bigr|^{-(n_{k\backslash b}+q)/2}$$
$$ =\Gamma_{lm} \Bigl( \bigl| S_{k \backslash b }^V \bigr| (1 + \kappa_{lm} + d_k^2 \cdot  \frac{q_{lm}\cdot n_{k\backslash b}}{n_{k\backslash b}+q_{lm}} \, \overline{Y}_{\{k \backslash b\}}^T   (S_{k \backslash b }^{V})^{-1} \overline{Y}_{\{k \backslash b\}}) \Bigr)^{-(n_{k\backslash b}+q)/2}$$
\noindent because $(d_k \, \overline{Y}_{\{ k \backslash b\}} -  \overline{Y}_{\{b_k\}})(d_k \, \overline{Y}_{\{ k \backslash b\}} -  \overline{Y}_{\{b_k\}})^T$ is a rank 1 update, where we define $S_{k \backslash b }^V = S_{\{k\backslash b \}} + I_{l=m}V_{\{ij\}}$, and 
\begin{align}
\kappa_{lm}= \frac{q_{lm}\cdot n_{k\backslash b}}{n_{k\backslash b}+q_{lm}} \, \Bigl(  \overline{Y}_{\{b_k\}}^T  (S_{k \backslash b }^{V})^{-1} \overline{Y}_{\{b_k\}}  - 2 \, d_k \cdot   \overline{Y}_{\{k \backslash b\}}^T   (S_{k \backslash b }^{V})^{-1} \overline{Y}_{\{b_k\}} \Bigr) \nonumber \\ 
=o \Bigl( d_k^2 \cdot  \frac{q_{lm} \cdot n_{k\backslash b}}{n_{k\backslash b}+q_{lm}} \, \overline{Y}_{\{k \backslash b\}}^T  (S_{k \backslash b }^{V})^{-1} \overline{Y}_{\{k \backslash b\}}) \Bigr)\label{eqn:little_o_cor}
\end{align}
where $q_{lm}=2$ for $l=m$ (ie, in $p_{11}$ and $p_{22}$) and $q_{lm}=1$ for $l\neq m$ (ie, in $p_{12}$ and $p_{21}$) by their definitions, and where $\overline{Y}_{\{b_k\}}$, the mean of the $Y_i's$ in the blocking set allocated to component $k$, refers to $Y_i$, $Y_j$, or their mean, $\overline{Y}_{\{ ij \}}$, depending on $k$ and whether used in the expression for $p_{11}$, $p_{12}$, $p_{21}$, or $p_{22}$. Additionally, $I_{l=m}$ is the indicator function for $l=m$, indices for $p_{lm}$, so the $V_{\{ ij \}}$ will not be present in the $p_{12}$ and $p_{21}$ expressions because $Y_i$ and $Y_j$ are split between the two components in their allocation and so have no variation attributable to them within a component. Notice that since the expressions are sums of squares, non-singular by assumption, they are positive and increasingly so as $d_k$ gets large.  
We Because $\kappa_{lm}$ is of order $d_k$, we conclude Equation (\ref{eqn:little_o_cor}) and that the dominant term for $p_{11}/\gamma$ is
\begin{align*}
\Gamma_{11} \, | S_{{\{2 \backslash b \}}} |^{\frac{-n_{2\backslash b}}{2}}    \Bigl( \Bigl| S_{{\{1 \backslash b \}}} +  & V_{\{ ij\}} \Bigr|\cdot \\
( d_1^2 \cdot & \frac{2\cdot n_{1\backslash b}}{n_{1\backslash b}+2}  \overline{Y}_{\{1 \backslash b\}}^T   \bigl(S_{{\{1 \backslash b \}}} + V_{\{ ij\}} \bigr)^{-1} \overline{Y}_{\{1 \backslash b\}}) \Bigr)^{\frac{-n_{1\backslash b}-2}{2}}
\end{align*}
and analogously for $p_{22}/\gamma$.  The dominant terms for $p_{12}/\gamma$ and $p_{21}/\gamma$ are $$\Gamma_{lm} \Bigl( | S_{{\{2 \backslash b \}}} | ( d_2^2 \cdot  \frac{1\cdot n_{2\backslash b}}{n_{2\backslash b}+1}  \overline{Y}_{\{2 \backslash b\}}^T   S_{{\{2 \backslash b \}}}^{-1} \overline{Y}_{\{2 \backslash b\}})  \Bigr)^{-(n_{2\backslash b}+1)/2} \;\;\;\;\;\;\;  $$  
$$ \;\;\;\;\;\;\; \cdot \Bigl( | S_{{\{1 \backslash b \}}} | ( d_1^2 \cdot   \frac{1\cdot n_{1\backslash b}}{n_{1\backslash b}+1}  \overline{Y}_{\{1 \backslash b\}}^T   S_{{\{1 \backslash b \}}}^{-1} \overline{Y}_{\{1 \backslash b \}}) \Bigr)^{-(n_{1\backslash b}+1)/2}$$
for $l,m\in \{1,2\}$, $l\neq m$, as appropriate.  Then, for increasing $d_1$ and $d_2$, because $d_1 = {\mathcal O}(d_2)$ and $n_{k\backslash b}>0$ for all $k$, we see $p_{12}/\gamma$ and $p_{21}/\gamma$ are $o(p_{11}/\gamma)$, so that $p_{12}=o(p_{11})$ and $p_{21}=o(p_{11})$.

Assume without loss of generality that $n_{1\backslash b}\geq n_{2\backslash b}$.  Then $p_{11} = {\mathcal O}(p_{22})$ or $p_{11} = o(p_{22})$ depending on if equality holds or not, respectively, because of $-n_{k\backslash b}$ in the exponent of $d_k$ with $k=1,2$ in the expressions for $p_{11}$ and $p_{22}$ where $d_1 = {\mathcal O}(d_2)$ by assumption.  Regardless, since $Cor(C_i,C_j) = (p_{11} p_{22}/\gamma^2 - p_{21} p_{12}/\gamma^2) / \bigl( (p_{11}/\gamma + p_{12}/\gamma^2) (p_{22}/\gamma^2 + p_{21}/\gamma^2)\bigr)$, the dominant terms will be $p_{11} p_{22}/\gamma^2$ in numerator and denominator, yielding 1 and giving the result. 
\end{proof}

\begin{proof}[{\bf Proof of Corollary \ref{corollary_cor}}]
Suppose $d_1$ is fixed with $d_2 \rightarrow \infty$ without loss of generality. Define $$\epsilon_{12} = \kappa_{12} \cdot \Bigl( | S_{{\{2 \backslash b \}}} | ( d_2^2 \cdot   \frac{1\cdot n_{2\backslash b}}{n_{2\backslash b}+1}  \overline{Y}_{\{1 \backslash b \}}^T   S_{{\{2 \backslash b \}}}^{-1} \overline{Y}_{\{1 \backslash b\}}) \Bigr)^{-(n_{2\backslash b}+1)/2}$$

$$\epsilon_{21} = \kappa_{21} \cdot \Bigl( | S_{{\{2 \backslash b \}}} | ( d_2^2 \cdot   \frac{1\cdot n_{2\backslash b}}{n_{2\backslash b}+1}  \overline{Y}_{\{2 \backslash b\}}^T   S_{{\{2 \backslash b \}}}^{-1} \overline{Y}_{\{2\backslash b\}}) \Bigr)^{-(n_{2\backslash b}+1)/2}$$

$$\epsilon_{22} = \kappa_{22} \cdot \Bigl( | S^V_{{2 \backslash b }} | ( d_2^2 \cdot   \frac{1\cdot n_{2\backslash b}}{n_{2\backslash b}+1}  \overline{Y}_{\{2 \backslash b\}}^T   \bigl(S^V_{2 \backslash b }\bigr)^{-1} \overline{Y}_{\{2\backslash b\}}) \Bigr)^{-(n_{2\backslash b}+2)/2}$$

\noindent
where we recycle the $\kappa_{lm}$ notation, which here stands for an expression-specific scaling, rather than additive, constant and is a function of $\Gamma_{lm}$, the allocation of the observation associated with the constant $d_1$ and the normalizing constant $\gamma$.  Then by the definitions of $p_{12}$, $p_{21}$, and $p_{22}$, 
$ p_{12}=\epsilon_{12} + o(\epsilon_{12})={\mathcal O} \bigl(d_2^{-n_{2\backslash b} -1}\bigr)$, $p_{21}=\epsilon_{21} + o(\epsilon_{21})={\mathcal O} \bigl(d_2^{-n_{2\backslash b} -1}\bigr)$ , while $p_{22}=\epsilon_{22} + o(\epsilon_{22})={\mathcal O} \bigl(d_2^{-n_{2\backslash b} -2}\bigr)=o(\epsilon_{12})$,
and where $\epsilon_{12}={\mathcal O}(\epsilon_{21})$.  We arbitrarily take $o(\epsilon_{12})$ as the representative lower order term in what follows.  Since $p_{11}+ p_{12}+ p_{21}+ p_{22} = 1$,
$$p_{11}=1-\epsilon_{12}-\epsilon_{21}- o(\epsilon_{12})$$ 
So we have $p_{11} + p_{21}=p_{\cdot 1}=1-\epsilon_{12} - o(\epsilon_{12})$ and $p_{11} + p_{12}=p_{1 \cdot }=1-\epsilon_{21} - o(\epsilon_{12})$, so that as $\epsilon_{12}$ and $\epsilon_{21}$ become small, $p_{\cdot 1} \cdot p_{1 \cdot} = 1-\epsilon_{12}-\epsilon_{21}-o(\epsilon_{12})$.
The numerator for the correlation calculation $Cor(C_i,C_j)$ is $p_{11} - p_{\cdot 1} \cdot p_{1 \cdot} = o(\epsilon_{12})$. But since the denominator for the calculation is $\sqrt{p_{\cdot 1} (1-p_{\cdot 1})p_{\cdot 2} (1-p_{\cdot 2})}={\mathcal O}(\epsilon_{12})$ so $Cor(C_i,C_j) = o(\epsilon_{12})/{\mathcal O}(\epsilon_{12}) \rightarrow 0$ as $d_2$ gets large.  

\end{proof}

\begin{proof}[{\bf Proof of Lemma \ref{lemma_invariant}}]
In \citet{liu_covariance_1994}, the authors show in Theorem 3.2 that the operator norm restricted to the space of mean 0 square integrable functions 
associated with the transition matrix for the move $C_{\{ b\}}^{(n)} \rightarrow C_{\{ b\}}^{(n+1)}$ using the ascending indices in Gibbs sampling is $\sup_{g,h} Cor \bigl(g(C_i), h(C_j ) \bigr)$ where $g \mbox{ and }h$ are functions on ${\mathbb R}$ under which 
$C_i$ and $C_j$ are square integrable.  By their Lemma 2.1 (found in \citet{yosida_functional_2012}), the transition matrix associated with the move using descending indices has an identical norm. The product of these transition matrices is self-adjoint.  Since the operator norm of the product of adjoint operators is the product of the norms, we have it equal to $\sup_{g,h} Cor \bigl(g(C_i), h(C_j ) \bigr)^2$ \citep{conway_course_2019}.  Since the operator norm of a self-adjoint operator is equal to its spectral radius, the convergence rate, that radius is $\sup_{g,h} Cor \bigl(g(C_i), h(C_j ) \bigr)^2$.  

Since $C_i$ and $C_j$ are binary, all functions on them can be expressed as linear ones.  Since correlation is invariant under linear transformation up to sign, $Cor \bigl(g(C_i), h(C_j )^2$ is constant for all $g,h$, yielding the equality $\sup_{g,h} Cor \bigl(g(C_i), h(C_j ) \bigr)^2 = Cor(C_i, C_j )^2$ and giving the result.
\end{proof}

\begin{proof}[{\bf Proof of Theorem \ref{theorem_main}.}]


Using similar reasoning to the proof of Theorem \ref{theorem_cor},
for any element $[{\textbf U}_{C_{\{b \}\backslash ij}}]_{l,l}$ along the diagonal of ${\textbf U}_{C_{\{b \}\backslash ij}}$ there is one term in the $K$ term product involving any $d_k$, $k\in \{1, \dots , K \}$, and that is $d_l$, the dominant portion of which is 

$$  \Bigl( | S_{{\{l \backslash b \}}} + V_{\{ ij \}} | ( d_l^2 \cdot  \frac{2\cdot n_{l\backslash b}}{n_{l\backslash b}+2}  \overline{Y}_{\{l \backslash b\}}^T   \bigl(S_{{\{l \backslash b \}}} + V_{\{ ij \}}\bigr)^{-1} \overline{Y}_{\{l \backslash b\}}) \Bigr)^{-(n_{l\backslash b}+2)/2}$$
For any term $[{\textbf U}_{C_{\{b \}\backslash ij}}]_{l,m}$ with $l\neq m$ there are two terms in the $K$ term product involving any $d_k$, $k\in \{1, \dots , K \}$, and those are $d_l$ and $d_m$, the dominant portions of which are:

$$ \Bigl( | S_{{\{l \backslash b \}}} | ( d_l^2 \cdot  \frac{1\cdot n_{l\backslash b}}{n_{l\backslash b}+1}  \overline{Y}_{\{l \backslash b\}}^T   S_{{\{l \backslash b \}}}^{-1} \overline{Y}_{\{l \backslash b\}})  \Bigr)^{-(n_{l\backslash b}+1)/2} \;\;\;\;\;\;\;  $$  
$$ \;\;\;\;\;\;\; \cdot \Bigl( | S_{{\{m \backslash b \}}} | ( d_m^2 \cdot   \frac{1\cdot n_{m\backslash b}}{n_{m\backslash b}+1}  \overline{Y}_{\{m \backslash b\}}^T   S_{{\{m \backslash b \}}}^{-1} \overline{Y}_{\{m \backslash b \}}) \Bigr)^{-(n_{m\backslash b}+1)/2}$$
Because $n_{k\backslash b}>0$ for all $k$ by assumption, and $K$ is fixed, we have that 

\begin{align}
\sum_{\{m:m\neq l \}} [{\textbf U}_{C_{\{b \}\backslash ij}}]_{l,m}= o([{\textbf U}_{C_{\{b \}\backslash ij}}]_{l,l} )
\label{eqn:diag_dom}
\end{align}
which holds for all $l\in \{1, \dots  , K \}$ and all $C_{\{b \}\backslash ij} \in {\mathcal C}_{\{b \}\backslash ij}$.  Since 
\begin{align}
\sum_{C_{\{b \}\backslash ij} \in {\mathcal C}_{\{b \}\backslash ij}} [{\textbf U}_{C_{\{b \}\backslash ij}}]_{l,m} = u^{\{b\}}_{l,m}
\end{align}
and set ${\mathcal C}_{\{b \}\backslash ij}$ is finite,

\begin{align}
\sum_{\{m:m\neq l \}} u^{\{b\}}_{l,m}= o(u^{\{b\}}_{l,l} ) 
\end{align}
which holds for all $l\in \{1, \dots  , K \}$.


If we sum elements of ${\textbf U}_{C_{\{b \}\backslash ij}}$ according to the partition induced by application of $I(\cdot \leq k')$ to $C_i$ and $C_j$, defining $C_i^{k'} = I(C_i \leq k')$ and $C_j^{k'} = I(C_j \leq k')$, we can calculate



$$Cor( C_i^{k'}, C_j^{k'} ) = \frac{p_{11}^{\{b\}} p_{22}^{\{b\}} - p_{12}^{\{b\}} p_{21}^{\{b\}}}{\Bigl(p_{11}^{\{b\}} +p_{21}^{\{b\}}\Bigr)\Bigl(p_{22}^{\{b\}} + p_{12}^{\{b\}}\Bigr)}$$
where



$$p_{11}^{\{b\}} = \gamma \sum_{l\leq k', \\m\leq k'} u^{\{b\}}_{l,m} \;\;\;\;\;\;\; p_{22}^{\{b\}} =  \gamma \sum_{l> k', \\m> k'} u^{\{b\}}_{l,m}$$

$$p_{21}^{\{b\}} =  \gamma \sum_{l> k', \\ m\leq k'} u^{\{b\}}_{l,m}  \;\;\;\;\;\;\;  p_{12}^{\{b\}} =  \gamma \sum_{l\leq k', \\m> k'} u^{\{b\}}_{l,m} $$

\noindent
But by Equation (\ref{eqn:diag_dom}), $p_{21}^{C_{\{b \}\backslash ij}}=o(p_{11}^{C_{\{b \}\backslash ij}})$ and $p_{21}^{C_{\{b \}\backslash ij}}=o(p_{22}^{C_{\{b \}\backslash ij}})$

$p_{21}^{\{b\}}=o(p_{11}^{\{b\}})$ and $p_{21}^{\{b\}}=o(p_{22}^{\{b\}})$

and equivalently for $p_{12}^{C_{\{b \}\backslash ij}}$ because ${\textbf U}_{C_{\{b \}\backslash ij}}$ is symmetric.  So the dominant products in the numerator and denominator are both $p_{11}^{C_{\{b \}\backslash ij}} p_{22}^{C_{\{b \}\backslash ij}}$, yielding 1 as the $d_k \rightarrow \infty$.  Since this holds for all ${\pmb C}_{\{b \}\backslash ij} \in {\mathcal C}_{\{b \}\backslash ij} $ 
%
%
\begin{align*}
\Biggl(\frac{p_{11}^{C_{\{b \}\backslash ij}} p_{22}^{C_{\{b \}\backslash ij}} -  p_{12}^{C_{\{b \}\backslash ij}}  p_{21}^{C_{\{b \}\backslash ij}}}{\bigl(p_{11}^{C_{\{b \}\backslash ij}} +p_{21}^{C_{\{b \}\backslash ij}}\bigr)\bigl(p_{22}^{C_{\{b \}\backslash ij}} + p_{12}^{C_{\{b \}\backslash ij}}\bigr)} \Biggr)^2 \leq  & \nonumber \\
\Biggl(\frac{(p_{11}^{\{b \}} - p_{1\cdot}^{\{b \}}p_{\cdot 1}^{\{b \}}  )}{((p_{11}^{\{b \}} + p_{12}^{\{b \}})(p_{21}^{\{b \}} + p_{22}^{\{b \}})) }\Biggr)^2 \leq  & \rho_{\tilde{\pi}} ({\pmb C}_{\{ij\}}) \leq  \rho_{\tilde{\pi}} ({\pmb C}_{\{b \}}) 
\end{align*}
%
%
giving $\rho_{\tilde{\pi}} ({\pmb C}_{\{b \}}) \rightarrow 1$.
\end{proof}

\end{document}